\documentclass[11pt]{amsart}
\usepackage{mathptmx}
\usepackage{amsmath}
\usepackage{amscd}
\usepackage{amssymb}
\usepackage{amsthm}
\usepackage{xspace}
\usepackage[all,tips]{xy}
\usepackage[dvips]{graphicx}
\usepackage{verbatim}
\usepackage{syntonly}
\usepackage{hyperref}
\usepackage{amsmath, amsthm, graphics, amssymb,fullpage,color, epsfig,url}
\usepackage{indentfirst}
\usepackage{esint}
\usepackage[normalem]{ulem}

\providecommand{\MR}{\relax\ifhmode\unskip\space\fi MR }

\providecommand{\href}[2]{#2}

\theoremstyle{plain}
\newtheorem{thm}{Theorem}

\newtheorem{lem}[thm]{Lemma}
\newtheorem{prop}[thm]{Proposition}

\theoremstyle{definition}
\newtheorem*{rem}{Remark}

\newcommand{\disp}{\displaystyle}

\DeclareMathOperator{\supp}{supp}

\newcommand{\eps}{\varepsilon}
\newcommand{\vp}{\varphi}

\newcommand{\al}{\alpha}
\newcommand{\be}{\beta}

\newcommand{\de}{\delta}

\newcommand{\te}{\theta}

\newcommand{\om}{\omega}
\newcommand{\Om}{\Omega}
\newcommand{\si}{\sigma}

\newcommand{\ol}{\overline}

\newcommand{\nid}{\noindent}

\newcommand{\iny}{\infty}
\newcommand{\del}{ \partial}
\newcommand{\su}{\subset}
\newcommand{\LP}{\Delta}
\newcommand{\gr}{\nabla}

\newcommand{\Lt}{\mathcal{L}_\tau}

\newcommand{\norm}[1]{\left\| #1\right\|}

\newcommand{\abs}[1]{\left\vert#1\right\vert}
\newcommand{\set}[1]{\left\{#1\right\}}
\newcommand{\brac}[1]{\left[#1\right]}
\newcommand{\pr}[1]{\left( #1 \right) }
\newcommand{\ceil}[1]{\lceil #1 \rceil }
\newcommand{\pb}[1]{\left( #1 \right] }
\newcommand{\bp}[1]{\left[ #1 \right) }
\newcommand{\der}[2]{\frac{\del #1}{\del #2} }

\newcommand{\N}{\ensuremath{\mathbb{N}}}

\newcommand{\R}{\ensuremath{\mathbb{R}}}
\newcommand{\Z}{\ensuremath{\mathbb{Z}}}
\newcommand{\C}{\ensuremath{\mathbb{C}}}

\newcommand{\RD}[1]{\textcolor{red}{#1}}

\title{Improved quantitative unique continuation for \\ complex-valued drift equations in the plane}
\author[Davey, Kenig, Wang]{Blair Davey \and Carlos Kenig \and Jenn-Nan Wang}
\address{Blair Davey, Department of Mathematics, City College of New York CUNY, New York, NY 10031, USA}
\email{bdavey@ccny.cuny.edu}
\thanks{Davey is supported in part by the Simons Foundation Grant number 430198.}
\address{Carlos Kenig, Department of Mathematics, University of Chicago, Chicago, IL 60637, USA}
\email{cek@math.uchicago.edu}
\thanks{Kenig is supported in part by NSF DMS-1800082}
\address{Jenn-Nan Wang, Institute of Applied Mathematical Sciences, NCTS, National Taiwan
University, Taipei 106, Taiwan}
\email{jnwang@math.ntu.edu.tw}
\thanks{Wang is supported in part by MOST 108-2115-M-002-002-MY3}

\subjclass[2010]{35J47, 35J10, 35J05}
\keywords{Carleman estimates, elliptic systems, quantitative unique continuation}
\date{}

\begin{document}

\begin{abstract}
In this article, we investigate the quantitative unique continuation properties of complex-valued solutions to drift equations in the plane.
We consider equations of the form $\LP u + W \cdot \gr u = 0$ in $\R^2$, where $W = W_1 + i W_2$ with each $W_j$ real-valued.
Under the assumptions that $W_j \in L^{q_j}$ for some $q_1 \in \brac{2, \iny}$, $q_2 \in \pb{2, \iny}$, and $W_2$ exhibits rapid decay at infinity, we prove new global unique continuation estimates.
This improvement is accomplished by reducing our equations to vector-valued Beltrami systems.
Our results rely on a novel order of vanishing estimate combined with a finite iteration scheme.
\end{abstract}

\maketitle

\section{Introduction}

The goal of this paper is to show that under suitable hypotheses, we may establish a stronger quantification of the unique continuation properties of complex-valued solutions to drift equations in $\R^2$ of the form
\begin{align}
- \LP u +  W \cdot \gr u = 0.
\label{ePDE}
\end{align}

Before describing our main results, we recall a few fundamental concepts in unique continuation theory.
The partial differential equation (PDE) $Lu = 0$ is said to have the {\em unique continuation property (UCP)} if whenever $u$ is a solution in $\Om$ and $u \equiv 0$ in an open subset of $\Om$, then $u \equiv 0$ in $\Om$.
Going further, the equation $L u = 0$ is said to have the {\em strong unique continuation property (SUCP)} if whenever $u$ is a solution in $\Om$ and $u$ vanishes to infinite order at some point $x_0 \in \Om$ (in an appropriate sense), then $u \equiv 0$ in $\Om$.
Therefore, whenever we are in a setting where the SUCP holds, it makes sense to ask the following question:
$$\textrm{What is the fastest rate of decay that a non-trivial solution can have?}$$
This local quantity is referred to as the {\em order of vanishing} and can be interpreted as a quantification of the SUCP.
A related global object is the {\em rate of decay at infinity}, a quantity that distinguishes between trivial and non-trivial entire solutions based on their asymptotic behavior.
Other topics of study in unique continuation theory include doubling indices and nodal (zero) sets of solutions.
We refer the reader to \cite{LM16, Log18a, Log18b} for recent progress in these related directions.
Our current work is related to Landis' conjecture, which seeks to determine the optimal rate of decay at infinity for solutions to Schr\"odinger equations.
As briefly described above, order of vanishing estimates are interesting on their own, but these quantities also serve as an important tool in our study of quantitative unique continuation at infinity properties.

In the late 1960s, E.~M.~Landis \cite{KL88} conjectured that if $u$ is a bounded solution to 
\begin{equation}
\label{ePDE2}
\LP u - V u = 0
\end{equation}
in $\R^n$, where $V$ is a bounded function and $\abs{u(x)} \lesssim \exp\pr{- c \abs{x}^{1+}}$, then $u \equiv 0$.
This conjecture was later disproved by Meshkov \cite{M92} who constructed non-trivial functions $u$ and $V$ that solve $\LP u - V u = 0$ in $\R^2$, where $V$ is bounded and $\abs{u(x)} \lesssim \exp\pr{- c \abs{x}^{4/3}}$. 
Meshkov also proved the following {\em qualitative unique continuation} result: 
If $\LP u - V u = 0$ in $\R^n$, where $V$ is bounded and $\abs{u\pr{x}} \lesssim \exp\pr{- c \abs{x}^{4/3+}}$, then necessarily $u \equiv 0$.

In their work on Anderson localization \cite{BK05}, Bourgain and Kenig established a quantitative version of Meshkov's result. 
As a first step in their proof, they used three-ball inequalities derived from a Carleman estimates to establish order of vanishing estimates for local solutions to Schr\"odinger equations.
Then, through a scaling argument, they showed that if $u$ and $V$ are bounded, and $u$ is normalized so that $\abs{u(0)} \ge 1$, then for sufficiently large values of $R$,
\begin{equation*}
 \inf_{|x_0| = R}\norm{u}_{L^\iny\pr{B_{1}(x_0)}} \ge \exp{(-CR^{4/3}\log R)}.
\end{equation*} 
Since $ \frac 4 3 > 1$, the constructions of Meshkov, in combination with the qualitative and quantitative unique continuation theorems just described, indicate that Landis' conjecture cannot be true for complex-valued solutions at least in $\R^2$.
However, Landis' conjecture still remains open in the general real-valued case.

In recent years, there has been a surge of activity surrounding Landis' conjecture in the real-valued planar setting.
The breakthrough article \cite{KSW15} proved a quantitative form of Landis' conjecture under the assumption that the zeroth-order term satisfies $V \ge 0$ a.e.
Subsequent papers established analogous results in the settings with variable coefficients \cite{DKW17} and singular lower order terms \cite{KW15, DW20}.
More recently, it has been shown that this theorem still holds when $V_-$ exhibits rapid decay at infinity \cite{DKW19}, and when $V_-$ exhibits slow decay at infinity \cite{Dav19a}.

The work in \cite{KW15} focuses on quantitative Landis-type theorems for {\bf real-valued} solutions to drift equations in the plane of the form \eqref{ePDE}.
One of the main theorems in \cite{KW15} shows that if $W \in L^q$ for some $q \in \brac{2, \iny}$ and $u$ is a real-valued, bounded, normalized solution to \eqref{ePDE}, then whenever $R$ is sufficiently large, it holds that 
\begin{equation}
\inf_{\abs{z_0} = R} \norm{u}_{L^\iny\pr{B_1\pr{z_0}}} \ge \left\{\begin{array}{ll} \exp\pr{- C R^{1 - \frac 2 q} \log R} & \text{ if } q > 2 \\ R^{-C} & \text{ if } q =2 \end{array} \right..
\label{realResult}
\end{equation}
In contrast, the article \cite{DZ18} contains quantitative Landis-type theorems for {\bf complex-valued} solutions to elliptic equations in the plane.
The related theorem in \cite{DZ18} for drift equations shows that if $W \in L^q$ for some $q \in \pb{2, \iny}$ and $u$ is a complex-valued, bounded, normalized solution to \eqref{ePDE}, then whenever $R$ is sufficiently large, it holds that 
\begin{equation}
\inf_{\abs{z_0} = R} \norm{u}_{L^\iny\pr{B_1\pr{z_0}}} \ge \exp\pr{- C R^{2} \log R}.
\label{complexResult}
\end{equation}
By comparing the results of \eqref{realResult} and \eqref{complexResult}, we see that the rate of decay significantly improves when we restrict to the real-valued setting.
In particular, the presence of an imaginary part of $W$ drastically affects the rate of decay of solutions.
This current paper is motivated by our desire to understand and quantify the effect that the complex part of $W$ has on the rate of decay at infinity.

In \cite{Dav14} and \cite{LW14}, the authors investigated the quantitative unique continuation properties of solutions to elliptic equations with lower order terms that exhibit pointwise decay at infinity.
The results in \cite{Dav14} and \cite{LW14} imply that if $W \in L^\iny$ exhibits rapid enough polynomial decay at infinity and $u$ is a complex-valued, bounded, normalized solution to \eqref{ePDE}, then whenever $\eps > 0$ and $R$ is sufficiently large, it holds that 
\begin{equation}
\inf_{\abs{{z_0}} = R} \norm{u}_{L^\iny\pr{B_1\pr{{z_0}}}} \ge \exp\pr{- R^{1+\eps}}.
\label{decayResult}
\end{equation}
We initiated this project with the belief that we could somehow combine the results described by \eqref{realResult}, \eqref{complexResult}, and \eqref{decayResult}.
As described in Theorem \ref{LandisThm} below, this is in fact true is we assume that the complex part of $W$ exhibits significant exponential decay at infinity in an appropriate sense that we will quantify.

In order to further understand the motivation for the current setting, we will describe the techniques that led to the estimates in \eqref{realResult}, \eqref{complexResult}, and \eqref{decayResult}.
Carleman estimate techniques were used in \cite{DZ18}, while Carleman estimates were combined with iterative arguments in \cite{Dav14, LW14} to prove \eqref{complexResult} and \eqref{decayResult}, respectively.
Such techniques have been used to prove many other results related to Landis' conjecture, see for example \cite{BK05, DZ19, Dav19b}.  
The Carleman method is applicable in any dimension and, in some cases, it gives rise to optimal bounds in the complex-valued setting.
Since Carleman estimates do not distinguish between real and complex values,  a different approach was used in \cite{KW15} to prove \eqref{realResult}, where the focus was on real-valued solutions and equations in the plane.
The proofs in \cite{KSW15, KW15, DKW17, DW20, DKW19, Dav19a} center around the relationship between second-order elliptic equations in the plane and Beltrami equations.
In suitable settings, one can use a second-order PDE to generate a Beltrami equation, a first-order elliptic equation in the complex plane.
The similarity principle for solutions to the Beltrami equation, along with Hadamard's three-circle theorem, leads to a three-ball inequality similar to the one derived in \cite{BK05}.
However, these new three-ball inequalities gives the precise exponents that could not be achieved with a direct Carleman approach.

In this article, by viewing complex-valued drift equations as systems of real-valued drift equations, we have found a way to combine many of the ideas mentioned above.
First we show that \eqref{ePDE} can be realized as a system of real-valued drift equations.
Then we show that such real-valued systems can be reduced to vector-valued Betrami equations.
Instead of invoking a similarity principle for these systems (as we did in \cite{DKW19}), we rely on $L^p - L^2$ Carleman estimates for the operator $\bar \del$ (similar to those that were previously developed in \cite{DLW19}) to give rise to our three-ball inequalities.
The three-ball inequality is then used to establish the order of vanishing result.
If the complex part of the potential function decays sufficiently quickly, then a scaling argument combined with repeated applications of the order of vanishing estimate gives rise to our quantitative unique continuation at infinity estimates.

Before stating the main result of this article, we describe the kinds of potential functions that we will work with.
Assume that there exist $q_1 \in \brac{2, \iny}$, $q_2 \in \pb{2, \iny}$, $c_0, \de_0 > 0$ so that $W = W_1 + i W_2$, where $W_i: \R^2 \to \R^2$ for $i = 1, 2$, and 
\begin{align}
& \norm{W_{1}}_{L^{q_1}\pr{\R^2} } \le 1
\label{W1Cond} \\
& \norm{W_2}_{L^{q_2}\pr{B_1\pr{z_0}}}
\le \exp\pr{- c_0 \abs{z_0}^{1 - \frac 2 {q_1} + \de_0}} \quad \forall z_0 \in \R^2.
\label{W2Cond}
\end{align}
In particular, the real part of $W$ satisfies the same hypotheses as it did in \cite{KW15}, while the complex part of $W$ must decay exponentially at a rate that depends on the properties of the real part of $W$.

Now we may state the main result of this article.
The following theorem is quantitative unique continuation at infinity estimate for solutions to \eqref{ePDE}, or a Landis-type theorem for complex-valued drift equations.

\begin{thm}
\label{LandisThm}
Assume that for some $q_1 \in \brac{2, \iny}$, $q_2 \in \pb{2, \iny}$, $c_0, \de_0 > 0$, $W = W_1 + i W_2 : \R^2 \to \C^2$ satisfies \eqref{W1Cond} and \eqref{W2Cond}.
Let $u: \R^2 \to \C$ be a solution to \eqref{ePDE} that is bounded and normalized in the sense that for some $t_0 \in \brac{1, 2}$,
\begin{align}
& \abs{u\pr{z}} \le \exp\pr{C_0 \abs{z}^{1 - \frac 2 {q_1}}}
\label{uBd} \\
& \norm{\gr u}_{L^{t_0}\pr{B_1\pr{0}}} \ge 1,
\label{normed}
\end{align}
where $t_0 < 2$ when $q_1 = 2$. 
Then for any $\eps > 0$ and any $R \ge \tilde R\pr{R_0, C_0, q_1, q_2, c_0, \de_0, t_0, \eps}$, it holds that
\begin{equation}
\inf_{\abs{z_0} = R} \norm{u}_{L^\iny\pr{B_1\pr{z_0}}} \ge \exp\pr{- R^{1+\eps}}.
\label{globalEst}
\end{equation}
\end{thm}

\begin{rem}
The value $R_0$ that appears in this theorem belongs to $\pr{0,1/e}$ and is a byproduct of the Carleman estimate that is used in our proofs.
\end{rem}

Compared to the results of \cite{KW15}, this rate of decay estimate is more rapid.
That is, when we allow for a non-trivial complex part of the potential, even a rapidly decaying part, the order of vanishing jumps from $1 - \frac 2 {q_1}$ to any value greater than $1$.
On the flipside, this rate of decay is a great improvement over the results of \cite{DZ18} since the power is far below $2$.
In summary, when we consider equations with a rapidly-decaying complex part of the potential, the resulting rate of decay for solutions falls in between the rates for equations with a purely real potential and equations with a singular complex potential.

This theorem and the Landis-type results in \cite{Dav19a} and \cite{DKW19} all give the same bound for the rate of decay at infinity.
In both \cite{Dav19a} and \cite{DKW19}, the setting is real-valued and the zeroth-order potential, $V$, has a negative part that decays at infinity.
In \cite{DKW19}, we assume that $V_- = \max\set{-V, 0}$ exhibits (rapid) exponential decay at infinity, quantitatively similar to the assumption that has been placed on $W_2$ in the current article.
In both the current article and \cite{DKW19}, we reduce our PDE to a Beltrami system of equations in which the multiplying factor is a $2 \times 2$ off-diagonal matrix.
To ensure that the non-trivial entries of the matrix are small enough for our techniques to work, we assume that some part of the potential ($V_-$ in \cite{DKW19}, $W_2$ here) is exponentially small. 
The same unique continuation estimate was shown to hold in \cite{Dav19a} when $V_-$ exhibits (slow) polynomial decay at infinity.
There, it is observed that if $V_-$ decays polynomially at infinity, then a positive multiplier exists and can be used to transform the PDE into a scalar-valued Beltrami equation.
By avoiding the vector-valued setting, we don't need to impose any further decay conditions on the potential functions.
In the current setting, we don't see how to avoid the vector-valued setting, either with the introduction of a positive multiplier or through some other technique.
As such, we impose the condition that $W_2$ exhibits rapid decay at infinity.

To prove our global theorem, we rely on the following order of vanishing estimate.
Although this theorem serves as an important tool in the proof of our first result, it also provides a quantification of the strong unique continuation property for local solutions to \eqref{ePDE}.
Furthermore, since this theorem allows the real part of $W$ to belong to $L^2$ instead of $L^{2+}$, then this result serves as an improvement over other known results in this direction, see for example \cite[Corollary 1]{DZ18}.
An alternative order of vanishing theorem appears below within Section \ref{OofVS}.

\begin{thm}
\label{OofV1}
Let $d \in \pb{1, 2}$.
Assume that for some $q_1 \in \brac{2, \iny}$ and $q_2 \in \pb{2, \iny}$, $\norm{W_j}_{L^{q_j}\pr{B_d}} \le M_j$ for $j = 1,2$.
Let $u$ be a solution to \eqref{ePDE} in $B_d$ that satisfies
\begin{align}
& \norm{u}_{L^\iny\pr{B_d}} \le \hat C. 
\label{localBd} 
\end{align}
If $q_1 > 2$ and we assume that
\begin{align}
& \norm{\gr u}_{L^2\pr{B_1}} \ge \hat c,
\label{localNorm}
\end{align}
 then for any $z_0 \in B_1$ and any $r$ sufficiently small, 
\begin{equation}
\norm{\gr u}_{L^2\pr{B_r\pr{z_0}}} \ge 
r^{C_{2} \brac{1+ M_2^{\mu_2} \exp\pr{C_3 M_1}} + \frac{c}{\log d} \set{C_{1} M_1 + \log\brac{\frac{C_2 \hat C \pr{1 + M_2} }{\hat c \sqrt{d-1}}}}},
\label{localEst}
\end{equation}
where $\mu_2 = \frac{2q_2}{q_2-2}$, $C_1 = C_1\pr{R_0, q_1}$, $C_2 = C_2\pr{R_0, q_2}$, $C_3 = C_3\pr{R_0, q_1, q_2}$, and $c$ is universal.
\\
If $q_1 = 2$ and we assume that for some $t_0 \in \bp{1, 2}$, 
\begin{align}
& \norm{\gr u}_{L^{t_0}\pr{B_1}} \ge \hat c,
\label{localNorm2}
\end{align} 
then for any $z_0 \in B_1$, any $r$ sufficiently small, any $q \in \pr{2, q_2}$, any $t \in \pr{\max\set{\frac{q}{q-1}, t_0}, 2}$, and any $t_1 \in \pb{t, 2}$,
\begin{equation}
\norm{\gr u}_{L^{t_1}\pr{B_r\pr{z_0}}} \ge r^{C_2 \brac{1+ M_2^{\mu} \exp\pr{C_3 M_1^2}} + \frac{c}{\log d} \set{C_1 M_1^2 + \log \brac{\frac{C_2 \hat C\pr{1 + M_2}}{ \hat c \sqrt{d-1}}}}},
\label{localEst2}
\end{equation}
where $\mu = \frac{t q }{t q- q - t}$, $C_1 = C_1\pr{R_0, q, t_0, t, t_1}$, $C_2 = C_2\pr{R_0, q_2, q, t}$, $C_3 = C_3\pr{R_0, q_2, q}$, and $c$ is universal.
\end{thm}

\begin{rem} 
If $W_2 \equiv 0$, then $M_2 = 0$ and we recover results on the order of vanishing estimates and the decay rates at infinity (a real version of Theorem~\ref{LandisThm}) from \cite{KW15}.
As such, this theorem may be interpreted as a complex perturbation of the real-valued result.
\end{rem}

The article is organized as follows.
In the next section, Section \ref{BelSysS}, three-ball inequalities for general vector-valued Beltrami systems are used to prove order-of-vanishing estimates for solutions to such equations.
Section \ref{OofVS} shows how the drift equation \eqref{ePDE} may be reduced to a vector-valued Beltrami equation.
Using these new presentations, we prove the order of vanishing results given by Theorems \ref{OofV1} and \ref{OofV0}.
Section \ref{UCEstS} shows how Theorem \ref{LandisThm} follows from Theorem \ref{OofV1} through rescaling combined with iteration.
When $q_1 > 2$, we must use the alternative order of vanishing estimate described by Theorem \ref{OofV0} to initiate the iterative process.
As such, this section has been divided into two parts, corresponding to the proof for $q_1 > 2$ and the proof for $q_1 = 2$.
The Carleman estimates that are crucial to the proof in Section \ref{BelSysS} are presented in Section \ref{CarEstS}. \\

\nid {\bf Acknowledgement.}
Part of this research was carried out while the first author was visiting the National Center for Theoretical Sciences (NCTS) at National Taiwan University.
The first author wishes to the thank the NCTS for their financial support and their kind hospitality during her visit to Taiwan.

\section{Estimates for general Beltrami systems}
\label{BelSysS}

Here we use three-ball inequalities derived from Carleman estimates to prove order of vanishing estimates for solutions to $2$-vector equations of the form 
\begin{equation}
\label{vecEqn}
\bar \del \vec{v} = G \vec{v},
\end{equation}
where $\vec v=(v_1,v_2)$ is some $2$-vector and $G$ is a $2 \times 2$ matrix function.
This is the major tool in proving our order of vanishing estimates for drift equations. The following Carleman estimate for first order operators is crucial to the arguments.
For a very similar estimate, we refer the reader to \cite[Theorem 3.1]{DLW19}.

\begin{thm}
\label{Carlp2}
Let $p \in \pb{1, 2}$.
There exists $R_0 \in \pr{0, 1/e}$ so that for any $\tau$ sufficiently large and any $u \in C^\iny_c\pr{B_{R_0} \setminus \set{0}}$, it holds that
\begin{equation}
\label{3.1}
\tau^\be \norm{\pr{r \log r}^{-1} e^{-\tau \phi(r)} u }_{L^2\pr{B_{R_0}}}
\le C \norm{r^{1 - 2/p} \pr{\log r} e^{-\tau \phi\pr{r}} \bar \del u}_{L^p\pr{B_{R_0}}},
\end{equation}
where $\phi\pr{r} = \log r + \tfrac 1 2 \log \pr{\log r}^2$, $\be = 1 - \frac 1 p$, and $C = C\pr{p,R_0}$.
\end{thm}

The technical proof of this theorem appears below in Section \ref{CarEstS}. For now, we use this Carleman estimate to prove the following lower bound, which is the main result of this section.

\begin{thm}
\label{lowerBndThm}
Let $a \in \pb{1, 2}$.
Define $v = \abs{v_1} + \abs{v_2}$, where $\vec v$ is a $2$-vector solution to \eqref{vecEqn} in $B_a$ with $\norm{G}_{L^q\pr{B_a}} \le M$ for some $q \in \pb{2, \iny}$.
Assume that for some $t \in \pb{\frac{q}{q-1}, 2}$ and some $\hat c \le 1 \le \hat C$, $\norm{v}_{L^t\pr{B_1}} \ge \hat c$ and $\norm{v}_{L^t\pr{B_a}} \le \hat C.$
Then for any $r_0$ sufficiently small and any $b \in \pr{1, a}$, it holds that
\begin{align*}
\norm{v }_{L^t\pr{B_{r_0}}}
&\ge r_0^{C \pr{1+ M^{\mu} } + c \log\pr{\frac{C \hat C}{\hat c}}/\log b},
\end{align*}
where $\mu = \frac{tq}{tq - q-t}$, $C = C\pr{q, t, R_0}$, and $c$ is universal.
\end{thm}

\begin{rem}
The theorem gives the best result (i.e. minimizes $\mu$) when we choose $t = 2$.
However, for technical reasons, there will be situations where we need $t < 2$.
Therefore, we present the very general result and choose $t$ appropriately in the proofs of our order of vanishing theorems. 
\end{rem}

\begin{proof}
Choose $r_0$ sufficiently small and $b \in \pr{1, a}$.
Let $K_1 = \set{r_0/2 \le \abs{{z}} \le r_0}$, $K_2 = \set{r_0 \le \abs{{z}} \le b}$, and $K_3 = \set{b \le \abs{{z}} \le a}$.
Set $K = K_1 \cup K_2 \cup K_3 \su B_{a} \setminus \set{0}$ and define $\chi \in C^\iny_0\pr{K}$ where $\chi \equiv 1$ on $K_2$ and $\supp \gr \chi = K_1 \cup K_3$.
Define $\vec u = \chi \vec v$, where $\vec v$ is the solution to $\bar \del \vec{v} = G \vec{v}$.

Since $q \in \pb{2, \iny}$, then for any $t \in \pb{\frac{q}{q-1}, 2}$ we have that $p := \frac{q t}{q+t} \in \pb{1, 2}$.
For each $j$, set $\tilde u_j\pr{z} = u_j\pr{\frac a {R_0} z}$ so that $\supp \tilde u_j \su B_{R_0} \setminus \set{0}$.
Then we may apply the Carleman estimate described by Theorem \ref{Carlp2} with $p$ as chosen to each $\tilde u_j$.
With $\tilde u = \abs{\tilde u_1} + \abs{\tilde u_2}$ and $\tilde K = \frac {R_0} a K {\su B_{R_0} \setminus \set{0}}$, we see that
\begin{align*}
\tau^\be \norm{\pr{r \log r}^{-1} e^{-\tau \phi(r)} \tilde u }_{L^2\pr{\tilde K}} 
&\le \tau^\be \sum_{j=1,2} \norm{\pr{r \log r}^{-1} e^{-\tau \phi(r)} \tilde u_j }_{L^2\pr{\tilde K}} \\
&\le C \sum_{j=1,2} \norm{r^{1 - 2/p} \pr{\log r} e^{-\tau \phi(r)} \bar \del \tilde u_j }_{L^p\pr{\tilde K}},
\end{align*}
where {$r = \abs{z}$ and} $\be = 1 - \frac 1 p = 1 - \frac 1 t - \frac 1 q = \mu^{-1}$. 
Define $\rho\pr{z} = \frac{R_0}{a} \abs{z} = \frac{R_0}{a} r$.
An application of H\"older (since $t \le 2$) and a change of variables shows that
\begin{align}
\label{CarEstRes}
\tau^\be \norm{\pr{{\rho} \log {\rho}}^{-1} e^{-\tau \phi({\rho})} u }_{L^t\pr{K}} 
&\le C \sum_{j=1,2} \norm{{\rho}^{1 - 2/p} \pr{\log {\rho}} e^{-\tau \phi({\rho})} \bar \del u_j }_{L^p\pr{K}},
\end{align}
where $C$ depends on $q$, $t$, $R_0$.

Note that by \eqref{vecEqn}
\begin{align*}
\bar \del u_j 
&= \bar \del \chi v_j + \chi \bar \del v_j
= \bar \del \chi v_j  + \chi \sum_{k=1,2} g_{j k} v_{k} 
= \bar \del \chi v_j  + \sum_{k=1,2} g_{j k} u_{k} .
\end{align*}
This equation combined with H\"older's inequality shows that for each $j = 1, 2$,
\begin{align*}
&\norm{{\rho}^{1 - 2/p} \pr{\log {\rho}} e^{-\tau \phi\pr{{\rho}}} \bar \del u_j}_{L^p\pr{K}} \\
\le& \sum_{k=1,2}\norm{{\rho}^{1 - 2/p} \pr{\log {\rho}} e^{-\tau \phi\pr{{\rho}}} g_{j k} u_{k}}_{L^p\pr{K}} 
+ \norm{{\rho}^{1 - 2/p} \pr{\log {\rho}} e^{-\tau \phi\pr{{\rho}}} \abs{\gr \chi} v_j}_{L^p\pr{K_1 \cup K_3}} \\
\le& \sum_{k=1,2} \norm{g_{j  k} }_{L^q\pr{K}} \norm{{\rho}^{1 - 1/p} \pr{\log {\rho}}}^2_{L^\iny\pr{K}} \norm{\pr{{\rho} \log {\rho}}^{-1} e^{-\tau \phi\pr{{\rho}}} u_k}_{L^t\pr{K}} \\
+& \norm{{\rho} \abs{\gr \chi}}_{L^\iny\pr{K_1}} \norm{{\rho}^{- 2/q} }_{L^q\pr{K_1}} \norm{{\rho}^{-2/t} \pr{\log {\rho}} e^{-\tau \phi\pr{{\rho}}} v_j}_{L^t\pr{K_1}} \\
+& \norm{\gr \chi}_{L^\iny\pr{K_3}} \norm{{\rho}^{1-2/q} }_{L^q\pr{K_3}} \norm{{\rho}^{-2/t} \pr{\log {\rho}} e^{-\tau \phi\pr{{\rho}}} v_j}_{L^t\pr{K_3}} . 
\end{align*}
A computation shows that$\norm{{\rho}^{1 - 1/p} \pr{\log {\rho}}}^2_{L^\iny\pr{K}}$, $\norm{{\rho} \abs{\gr \chi}}_{L^\iny\pr{K_1}}$, $\norm{{\rho}^{- 2/q} }_{L^q\pr{K_1}}$, and $\norm{\gr \chi}_{L^\iny\pr{K_3}} \norm{{\rho}^{1-2/q} }_{L^q\pr{K_3}}$ are bounded by constants depending on {$R_0$ and} $q$.
Combining the previous inequality with \eqref{CarEstRes} then shows that
\begin{align*}
\tau^\be \norm{\pr{{\rho} \log {\rho}}^{-1} e^{-\tau \phi({\rho})} u }_{L^t\pr{K}} 
&\le C M \norm{\pr{{\rho} \log {\rho}}^{-1} e^{-\tau \phi\pr{{\rho}}} u}_{L^t\pr{K}}
+ C  \norm{{\rho}^{-2/t} \pr{\log {\rho}} e^{-\tau \phi\pr{{\rho}}} v}_{L^t\pr{K_1 \cup K_3}}.
\end{align*}
If $\tau \ge \pr{2 C M}^{\mu}$, then the first term may be absorbed into the left to get
\begin{align*}
\norm{e^{-\pr{\tau+1} \phi({\rho})} v }_{L^t\pr{K_2}}
\le& \norm{e^{-\pr{\tau+1} \phi({\rho})} \chi v }_{L^t\pr{K}}
\le \norm{e^{-\pr{\tau+1} \phi({\rho})} u }_{L^t\pr{K}} \\
\le& C \rho_0^{1 - 2/t}\pr{\log {\rho_0}}^2 \norm{e^{-\pr{\tau+1} \phi\pr{{\rho}}} v}_{L^t\pr{K_1}}
+ C \pr{\log R_0}^2 \norm{e^{-\pr{\tau+1} \phi\pr{{\rho}}} v}_{L^t\pr{K_3}},
\end{align*}
where we have used the definition of $\phi$ and introduced $\rho_0 := {R_0 r_0}/ a$.
Replacing $\tau+1$ with $\tau$ and assuming that $\tau \ge C\pr{1 + M^{\mu}}$, it holds that
\begin{align*}
\norm{ v }_{L^t\pr{\set{r_0 \le \abs{x} \le 1}}} 
\le& e^{\tau \phi(R_0/a)}  \norm{e^{-\tau \phi(\rho)} v }_{L^t\pr{K_2}}  \\
\le& C e^{\tau \phi(R_0/a)} \brac{\rho_0^{1 - 2/t} \pr{\log \rho_0}^2 \norm{e^{-\tau \phi\pr{\rho}} v }_{L^t\pr{K_1}}
+\pr{\log R_0}^2\norm{e^{-\tau \phi\pr{\rho}} v }_{L^t\pr{K_3}}} \\
\le& C \rho_0^{1 - 2/t} \pr{\log \rho_0}^2 \frac{e^{\tau \phi(R_0/a)}}{e^{\tau \phi\pr{\rho_0/2}}} \norm{v }_{L^t\pr{K_1}}
+ C\pr{\log R_0}^2 \frac{e^{\tau \phi(R_0/a)}}{e^{\tau \phi\pr{R_0 b/a}}} \norm{v }_{L^t\pr{K_3}}.
\end{align*}
Adding $\norm{v}_{L^t\pr{B_{r_0}}}$ to both sides of the inequality shows that
\begin{align*}
\norm{ v }_{L^t\pr{B_{1}}} 
\le& C \rho_0^{1 - 2/t} \pr{\log \rho_0}^2 e^{\tau\pr{\phi(R_0/a)-\phi\pr{\rho_0/2}}} \norm{v }_{L^t\pr{B_{r_0}}}
+ C\pr{\log R_0}^2 e^{\tau\pr{\phi(R_0/a)-\phi\pr{R_0b/a}}} \norm{v }_{L^t\pr{B_{a}}}.
\end{align*}
Define $\kappa = \frac{\phi\pr{R_0 b/a} - \phi\pr{R_0/a}}{\phi\pr{R_0 b/a} - \phi\pr{\rho_0/2}}$ and set
$$\tau_0 = \frac{\kappa}{\phi\pr{R_0 b/a} - \phi\pr{R_0/a}} \log\brac{\frac{\pr{\log R_0}^2\norm{v }_{L^t\pr{B_{a}}}}{\rho_0^{1 - 2/t}\pr{\log \rho_0}^2\norm{v }_{L^t\pr{B_{r_0}}}}}.$$
If $\tau_0 \ge C\pr{1 + M^{\mu}}$, then the above computations are valid with this choice of $\tau$ and we see that
\begin{align*} 
\norm{ v }_{L^t\pr{B_{1}}} 
\le& C \brac{\rho_0^{1 - 2/t}\pr{\log \rho_0}^2\norm{v }_{L^t\pr{B_{r_0}}}}^\kappa \brac{\pr{\log R_0}^2 \norm{v }_{L^t\pr{B_{a}}}}^{1 - \kappa}.
\end{align*}
On the other hand, if $\tau_0 < C\pr{1 + M^{\mu}}$, then
\begin{align*}
\norm{v }_{L^t\pr{B_1}}
\le \norm{v }_{L^t\pr{B_{a}}}
&\le \exp\brac{C\pr{1+ M^{\mu} }\pr{\phi\pr{R_0 b/a} - \phi\pr{\rho_0/2}}} {\rho_0}^{1 - 2/t} \pr{\frac{\log \rho_0}{\log R_0}}^2 \norm{v }_{L^t\pr{B_{r_0}}}.
\end{align*}
Adding the previous two inequalities and invoking the assumptions that $\hat c \le \norm{v }_{L^t\pr{B_1}}$ and $\norm{v }_{L^t\pr{B_a}} \le \hat C$ shows that
\begin{align*}
\hat c &\le {\rm I} + \Pi,
\end{align*}
where
\begin{align*}
{\rm I} &= C \brac{\rho_0^{1 - 2/t}\pr{\log \rho_0}^2\norm{v }_{L^t\pr{B_{r_0}}}}^\kappa \brac{ \pr{\log R_0}^2 \hat C}^{1 - \kappa} \\
\Pi &=  \exp\brac{C\pr{1+ M^{\mu} }\pr{\phi\pr{R_0 b/a} - \phi\pr{\rho_0/2}}}\rho_0^{1 - 2/t} \pr{\frac{\log \rho_0}{\log R_0}}^2 \norm{v }_{L^t\pr{B_{r_0}}}.
\end{align*}

On one hand, if ${\rm I} \leq \Pi$, then $\hat c \le 2 \Pi$ so that
\begin{align*}
\norm{v }_{L^t\pr{B_{r_0}}}
&\ge \frac {\hat c}{2}{\rho_0}^{2/t-1} \pr{\frac{\log R_0}{\log \rho_0}}^2  \exp\brac{C\pr{1+ M^{\mu} }\pr{\phi\pr{\rho_0/2} - \phi\pr{R_0 b/a} }} 
\end{align*}
Assuming that $r_0 \ll R_0$,
$$\phi\pr{\rho_0/2} - \phi\pr{R_0 b/a} \ge c \log r_0$$
and then
\begin{align}
\label{firstBound}
\norm{v }_{L^t\pr{B_{r_0}}}
&\ge C \hat c (\log R_0)^2 r_0^{C \pr{1+  M^{\mu} }}.
\end{align}

On the other hand, if $\Pi \leq {\rm I}$, then
\begin{align*}
\hat c
&\le 2 C \brac{\rho_0^{1 - 2/t} \pr{\log \rho_0}^2\norm{v }_{L^t\pr{B_{r_0}}}}^\kappa \brac{  \pr{\log R_0}^2 \hat C}^{1 - \kappa}.
\end{align*}
Raising both sides to $\frac{1}{\kappa}$ shows that
\begin{align*}
\norm{v }_{L^t\pr{B_{r_0}}} 
&\ge \hat C \rho_0^{2/t-1} \pr{\frac{\log R_0}{\log \rho_0}}^2 \brac{\frac{2C\hat C \pr{\log R_0}^2} {\hat c}}^{-1/\kappa}.
\end{align*}
As above, for any $r_0 \ll R_0$, $- \frac 1 \kappa =  \frac{\phi\pr{\rho_0/2} - \phi\pr{R_0 b/a}}{\phi\pr{R_0 b/a} - \phi\pr{R_0/a}} \ge \frac{c \log r_0}{\log b}$ and then
\begin{align}
\norm{v }_{L^t\pr{B_{r_0}}} 
&\ge \hat{C} (\log R_0)^2 r_0^{c \log \left[\frac{2C\hat C  \pr{\log R_0}^2}{\hat c}\right]/\log b}.
\label{secondBound}
\end{align}
Combining \eqref{firstBound} and \eqref{secondBound} leads to the conclusion of Theorem~\ref{lowerBndThm}.
\end{proof}

\section{Order of Vanishing Estimates}
\label{OofVS}

This section contains the proofs of our order of vanishing results, Theorem \ref{OofV1} in the introduction and Theorem \ref{OofV0} below.
The idea underlying our proofs is that we can reduce the PDE given in \eqref{ePDE} to a first-order Beltrami equation.
The novelty here is that the resulting equation is a vector equation instead of a scalar equation as it was in \cite{KSW15} and \cite{KW15}.
More specifically, we will show that the elliptic PDE described by \eqref{ePDE} is equivalent to an equation of the form \eqref{vecEqn}. 

If $u = u_1 + i u_2$, then the drift equation \eqref{ePDE} is equivalent to the system
\begin{equation}
\label{PDESystem}
\left\{ \begin{array}{l}  
\LP u_1 = W_1 \cdot \gr u_1 - W_2 \cdot \gr u_2 \\
\LP u_2 = W_1 \cdot \gr u_2 + W_2 \cdot \gr u_1. 
\end{array}\right.
\end{equation}
Recall that $\bar \del = \frac{\del}{\del \bar z} = \frac{1}{2}\pr{ \der{}{x} + i \der{}{y} }$ and $\del = \frac{\del}{\del  z} = \frac{1}{2}\pr{ \der{}{x} - i \der{}{y} }$.
Using the natural association between $2$-vectors and complex values, i.e. $\pr{a, b} \sim a + ib$, we define
\begin{align*}
W_k\pr{u_j} = \left\{\begin{array}{ll} \frac 1 4 \pr{W_k + \ol{W_k}\frac{\bar \del u_j}{\del u_j}} & \text{ if } \del u_j \ne 0 \\ 0 & \text{ otherwise } \end{array}  \right.
\end{align*}
so that
$$4 W_k\pr{u_j} \del u_j = W_k \del u_j + \ol{W_k} \bar \del u_j =2 \Re W_k \del u_j = W_k \cdot \gr u_j.$$
Then the system \eqref{PDESystem} may be rewritten as
\begin{align*}
\left\{ \begin{array}{l}  
\bar \del \del u_1 - W_1\pr{u_1} \del u_1 = - W_2\pr{u_2} \del u_2 \\
\bar \del \del u_2 - W_1\pr{u_2} \del u_2 = W_2\pr{u_1} \del u_1. 
\end{array}\right.
\end{align*}
If we define 
\begin{equation}
\label{vGDef0}
\vec{v} = \brac{\begin{array}{c} \del u_1 \\ \del u_2 \end{array}} \quad \text{ and } \quad 
G = \brac{\begin{array}{cc} W_1\pr{u_1} & - W_2\pr{u_2} \\ W_2\pr{u_1}  & W_1\pr{u_2} \end{array}},
\end{equation}
then the system of equations described by \eqref{PDESystem} is equivalent to \eqref{vecEqn}.

The following theorem is an alternative order of vanishing estimate.
Although Theorem \ref{OofV1} is our main order of vanishing estimate, we will use the following result to initiate the proof of Theorem \ref{LandisThm} in the setting where $q_1 > 2$.
This proof is also interesting because it demonstrates how we make use of the Beltrami representation in a simpler setting.

\begin{thm}
\label{OofV0}
Assume that for some $q \in \pb{2, \iny}$, $\norm{W}_{L^{q}\pr{B_2}} \le M$.
Let $u$ be a solution to \eqref{ePDE} in $B_2$ that satisfies \eqref{localBd} with $d = 2$ and \eqref{localNorm}.
Then for any $r$ sufficiently small, 
\begin{equation}
\norm{\gr u}_{L^2\pr{B_r}} \ge r^{C \pr{1+ M^{\mu} } + c \log \brac{\frac{C\hat C\pr{1 + M}}{\hat c}}},
\label{localEst10}
\end{equation}
where $\mu = \frac{2q}{q-2}$, $C = C\pr{q, R_0}$. 
\end{thm}

\begin{rem}
An application of the Cacciopoli inequality as in \eqref{intEst} below allows us to replace the $L^2$-norm of the gradient on the lefthand side with the $L^\iny$-norm of the function itself.
After such a reduction, this result is essentially the same as the order of vanishing result from \cite[Corollary 1]{DZ19}.
The proof that we present here is different.
\end{rem}

\begin{rem}
Consider the case with $q = \iny$.
Then $\mu = 2$ and we obtain the well-known order of vanishing estimate for drift equations, see for example \cite{Dav14}.
\end{rem}

\begin{rem}
This theorem differs from Theorem \ref{OofV1} and, at first glance, it may appear that this theorem is stronger because of the absence of an exponential dependence in the bound.
However, this theorem doesn't cover the case of $q_1 = 2$.
Moreover, if $M_2 \ll M_1$, then the bound that we obtain in Theorem \ref{OofV1} is better than this one.
In a sense, our new result may be interpreted as a perturbation of the order of vanishing results for real-valued solutions to drift equations that appeared in \cite{KW15}.
This theorem holds for complex-valued equations.
\end{rem}

\begin{proof}
If we define $\vec v$ and $G$ as in \eqref{vGDef0}, then equation \eqref{vecEqn} holds in $B_2$.
With $v = \abs{v_1} + \abs{v_2}$, we see that $v \sim \abs{\gr u}$.
Therefore, it follows from \eqref{localNorm} that $\norm{v}_{L^2\pr{B_1}} \gtrsim \hat c$.
By the assumption on $W$ and the fact that $\abs{W_j\pr{u_k}\pr{z}} \le \abs{W_j\pr{z}}$ for all $z$,  we see that $\norm{G}_{L^q\pr{B_2}} \le C M$.
A standard integration by parts argument shows that whenever $\LP u = W \cdot \gr u$ in $B_{R}$, where $W \in L^q\pr{B_R}$ for some $q \in \brac{2, \iny}$,
\begin{equation}
\label{intEst}
\norm{\gr u}_{L^2\pr{B_r}} \le C \brac{\pr{1-\frac r R}^{-1/2} + R^{1 - \frac 2 q}\norm{W}_{L^q\pr{B_{R}}}} \norm{u}_{L^\iny\pr{B_{R}}}.
\end{equation}
Combining \eqref{intEst} with \eqref{localBd} then implies that $\norm{v}_{L^2\pr{B_{3/2}}} \le \hat C \pr{1 + M}$.
An application of Theorem~\ref{lowerBndThm} with $t = 2$ and $a = 3/2$ shows that
\begin{align*}
\norm{\gr u}_{L^2\pr{B_{r}\pr{x_0}}}
\gtrsim \norm{v }_{L^2\pr{B_{r}\pr{x_0}}}
&\ge r^{C \pr{1+ M^{\mu} } + c \log \brac{\frac{C\hat C\pr{1 + M}}{\hat c}}},
\end{align*}
as required.
\end{proof}

Returning to the Beltrami system from \eqref{vGDef0} and the preceding line, we take an alternative approach and define 
\begin{equation}
\label{vjDef}
v_j = \del u_j e^{-T\pr{W_1\pr{u_j}}} \quad \text{for each} \quad j = 1, 2,
\end{equation}
where we use the notation $T = T_{B_d}$ to denote the Cauchy-Pompeiu operator on $B_d$.
Then
\begin{align*}
\bar \del v_j 
&= \bar \del \pr{\del u_j e^{-T\pr{W_1\pr{u_j}}}}
= \brac{\bar \del \del u_j - W_1\pr{u_j} \del u_j} e^{-T\pr{W_1\pr{u_j}}} \\
&= \pr{-1}^j W_2\pr{u_{\hat j}} \del u_{\hat j} e^{-T\pr{W_1\pr{u_j}}}
= \pr{-1}^j W_2\pr{u_{\hat j}} e^{T\brac{W_1\pr{u_{\hat j}}- W_1\pr{u_j}}} v_{\hat j},
\end{align*}
where $\hat j = j \pm 1$.
If we introduce the vector notation
\begin{equation}
\label{vGDef}
\vec{v} = \brac{\begin{array}{c} v_1 \\ v_2 \end{array}} \quad \text{ and } \quad 
G = \brac{\begin{array}{cc} 0 & - W_2\pr{u_2} e^{T\brac{W_1\pr{u_2}- W_1\pr{u_1}}}  \\ W_2\pr{u_1} e^{-T\brac{W_1\pr{u_{2}}- W_1\pr{u_1}}}  & 0 \end{array}},
\end{equation}
then \eqref{vecEqn} holds.
This is the representation that will be used in the proof of our order of vanishing estimate described by Theorem \ref{OofV1}.

Before proving that theorem, we establish an $L^q$-bound for the matrix $G$ given in \eqref{vGDef}.
To do this, we have to recall some properties of $T$. Let $\om \in L^q$ for some $q \in \brac{2, \iny}$ satisfy $\norm{\om}_{L^q\pr{B_d}} \le M$.
The Cauchy-Pompeiu transform of $\om$ is defined as
\begin{equation*}
T {\om} (z)=\frac{1}{\pi}\int_{B_d}\frac{\om(\xi)}{\xi-z}d\xi.
\end{equation*}
If $q > 2$, then $T\pr{\om} \in L^\iny$ with $\norm{T{\om}}_{L^\iny\pr{B_d}} \le C M$, where $C$ depends on $q$ and $d$.
Otherwise, if $q = 2$, then $T\pr{\om} \in W^{1,2}$ with
$$\norm{T {\om}}_{W^{1,2}\pr{B_d}}=\norm{T {\om}}_{L^2\pr{B_d}}+ \norm{\gr {T {\om}}}_{L^2\pr{B_d}}\le CM.$$
For further analysis of $T\om$ in the setting where $q = 2$, we recall the following lemma from \cite{KW15}.

\begin{lem}[cf. Lemma 3.3 in \cite{KW15}]
\label{lemma0701}
Set $h = T \om$ for some $\om \in L^2\pr{B_d}$ with $\norm{\om}_{L^2\pr{B_d}} \le M$.
For $s>0$ and $0<r\le d$, it holds that
\begin{equation}
\label{qqest}
\fint_{B_r}\exp(s|h|)\le Cr^{-sCM}\exp(sCM+s^2CM^2),
\end{equation}
where we denote $\disp \fint_{B_r} f=|B_r|^{-1}\int_{B_r}f$. 
\end{lem}

Now we can show that $G$ is bounded in $L^q$ for some $q \in \pb{2, q_2}$.

\begin{lem}
\label{GBound}
Assume that $d\in(1,2]$ and for some $q_1 \in \brac{2, \iny}$ and $q_2 \in \pb{2, \iny}$, $\norm{W_j}_{L^{q_j}\pr{B_d}} \le M_j$ for $j = 1,2$.
Define the matrix function $G$ as in \eqref{vGDef}.
Set $q = q_2$ if $q_1 > 2$ and otherwise choose $q \in \pr{2, q_2}$.
Then
\begin{align*}
\norm{G}_{L^q\pr{B_d}} \lesssim M_2 \exp\pr{C M_1^\al},
\end{align*}
where $\al = 1$ if $q_1 > 2$ and $\al = 2$ otherwise.
\end{lem}

\begin{proof}
Recall that $G_{jj} = 0$ and $G_{j \hat j} = \pr{-1}^{ j} W_2\pr{u_{\hat j}} e^{\pr{-1}^{\hat j}T\brac{W_1\pr{u_2}- W_1\pr{u_1}}}$.
Since $\abs{W_j\pr{u_k}\pr{z}} \le \abs{W_j\pr{z}}$ for all $z$, then $W_j \in L^{q_j}$ implies that $W_j\pr{u_k} \in L^{q_j}$ as well with the same norm.

If $q_1 > 2$, then
$$\norm{T\brac{W_1\pr{u_2}- W_1\pr{u_1}}}_{L^\iny\pr{B_d}} \le C M_1$$
and then
$$\norm{G}_{L^{q_2}\pr{B_d}} \le M_2 \exp\pr{C M_1}.$$
If $q_1 = 2$, choose $q \in \pr{2, q_2}$ and set $s = \frac{q q_2}{q_2 - q}$.
An application of the H\"older inequality shows that
\begin{align*}
\norm{G_{j \hat j}}_{L^{q}\pr{B_d}} 
&= \norm{W_2\pr{u_{\hat j}} e^{\pr{-1}^{\hat j}T\brac{W_1\pr{u_2}- W_1\pr{u_1}}}}_{L^{q}\pr{B_d}} 
\le \norm{W_2}_{L^{q_2}\pr{B_d}} \norm{e^{T\brac{W_1\pr{u_2}- W_1\pr{u_1}}}}_{L^{s}\pr{B_d}} \\
&\le C_s d^{2/s} M_2 \pr{ \fint_{B_d} \exp\pr{s\abs{T\brac{W_1\pr{u_2}- W_1\pr{u_1}}}}}^{1/s} \\
&\le C_{s} d^{-CM_1}M_2 \exp\pr{C M_1 +s C M_1^2},
\end{align*}
where the last step invokes Lemma~\ref{lemma0701}.
The conclusion follows.
\end{proof}

Now we prove the new order of vanishing estimate described by Theorem~\ref{OofV1}.

\begin{proof}[Proof of Theorem~\ref{OofV1}]
Define $\vec v$ and $G$ as in \eqref{vjDef} and \eqref{vGDef} so that equation \eqref{vecEqn} holds in $B_d$.
Choose $1 < b < a < d$ so that $b - 1 \simeq a - b \simeq d - a$.
Then $\log b \simeq \log d$ and $a - b \simeq d -1$.
Set $v = \abs{v_1} + \abs{v_2}$.
In order to keep track of the dependencies in the constants, we'll use a subscript notation within this proof.

Assume first that $q_1 > 2$.
We see from \eqref{localNorm} and H\"older's inequality that 
\begin{align*}
\hat c &\le \norm{\gr u}_{L^2\pr{B_1}}
\le \norm{\gr u_1}_{L^2\pr{B_1}} 
+ \norm{\gr u_2}_{L^2\pr{B_1}} 
= \norm{ e^{T\pr{W_1\pr{u_1}}} v_1}_{L^2\pr{B_1}} 
+ \norm{e^{T\pr{W_1\pr{u_2}}} v_2}_{L^2\pr{B_1}} \\
&\le \norm{ e^{T\pr{W_1\pr{u_1}}}}_{L^\iny\pr{B_1}} \norm{ v_1}_{L^2\pr{B_1}} 
+ \norm{e^{T\pr{W_1\pr{u_2}}}}_{L^\iny\pr{B_1}} \norm{v_2}_{L^2\pr{B_1}}
\le \exp\pr{C_{q_1} M_1} \norm{v}_{L^2\pr{B_1}}.
\end{align*}
It follows that $\norm{v}_{L^2\pr{B_1}} \ge \hat c \exp\pr{- C_{q_1} M_1}$. 
Similarly,
\begin{align*}
\norm{v}_{L^2\pr{B_a}}
&\le \norm{e^{-T\pr{W_1\pr{u_1}}} \gr u_1}_{L^2\pr{B_a}}
+ \norm{e^{-T\pr{W_1\pr{u_2}}} \gr u_2}_{L^2\pr{B_a}}
\le \exp\pr{C_{q_1} M_1} \norm{\gr u}_{L^2\pr{B_a}} \\
&\le \pr{\sqrt{\frac{d}{d-a}} + C_{q_1} M_1 + C_{q_2} M_2} \exp\pr{C_{q_1} M_1} \norm{u}_{L^\iny\pr{B_d}} 
\le \frac{\hat C \pr{1 + C_{q_2} M_2}}{\sqrt{d-1}} \exp\pr{C_{q_1} M_1} ,
\end{align*}
where we have applied the interior estimate described by \eqref{intEst} and the upper bound from \eqref{localBd}.
Since Lemma \ref{GBound} shows that $\norm{G}_{L^{q_2}\pr{B_d}} \le M_2 \exp\pr{C_{q_1} M_1}$, then an application of Theorem~\ref{lowerBndThm} with $t = 2$ shows that
\begin{align*}
\norm{v }_{L^2\pr{B_{r}\pr{x_0}}}
&\ge r^{C_{q_2} \set{1+ \brac{M_2 \exp\pr{C_{q_1} M_1}}^{\mu_2} } + \frac{c}{\log d} \set{C_{q_1} M_1 + \log\brac{\frac{C \hat C \pr{1 + C_{q_2} M_2}  }{\hat c \sqrt{d-1}}}}}.
\end{align*}
Since $\norm{v}_{L^2\pr{B_r}} \le \exp\pr{C_{q_1} M_1} \norm{\gr u}_{L^2\pr{B_r}}$, then we can rearrange to reach the conclusion of the theorem for the case $q_1 > 2$.

Now we consider $q_1 = 2$.
Choose $q \in \pr{2, q_2}$ and $t \in \pr{\max\set{\frac{q}{q-1}, t_0}, 2}$, then define $t' < \iny$ to satisfy $\frac 1 {t_0} = \frac 1 t + \frac 1 {t'}$.
It follows from the lower bound in \eqref{localNorm2} and H\"older's inequality that 
\begin{align*}
\hat c &\le \norm{\gr u}_{L^{t_0}\pr{B_1}}
\le \norm{\gr u_1}_{L^{t_0}\pr{B_1}} 
+ \norm{\gr u_2}_{L^{t_0}\pr{B_1}}  \\
&\le \norm{ e^{T\pr{W_1\pr{u_1}}}}_{L^{t'}\pr{B_1}} \norm{ v_1}_{L^t\pr{B_1}} 
+ \norm{e^{T\pr{W_1\pr{u_2}}}}_{L^{t'}\pr{B_1}} \norm{v_2}_{L^t\pr{B_1}}
\le \exp\pr{C_{t'} M_1^2} \norm{v}_{L^t\pr{B_1}},
\end{align*}
where we have applied Lemma \ref{lemma0701}.
Similarly,
\begin{align*}
\norm{v}_{L^t\pr{B_a}}
&\le \norm{e^{-T\pr{W_1\pr{u_1}}} \gr u_1}_{L^t\pr{B_a}}
+ \norm{e^{-T\pr{W_1\pr{u_2}}} \gr u_2}_{L^t\pr{B_a}}
\le \exp\pr{C_t M_1^2} \norm{\gr u}_{L^2\pr{B_a}} \\
&\le \pr{\sqrt{\frac d {d-a}} + C_{2} M_1 + C_{q_2} M_2} \exp\pr{C_t M_1^2} \norm{u}_{L^\iny\pr{B_d}}
\le \frac{\hat C \pr{1 + C_{q_2} M_2}}{\sqrt{d-1}} \exp\pr{C_t M_1^2}.
\end{align*}
Since Lemma \ref{GBound} implies that $\norm{G}_{L^{q}\pr{B_d}} \le M_2 \exp\pr{C_{q, q_2} M_1^2}$, then an application of Theorem~\ref{lowerBndThm} with our choice of $t$ shows that
\begin{align*}
\norm{v }_{L^t\pr{B_{r}\pr{x_0}}}
&\ge r^{C_{q,t} \set{1+ \brac{M_2 \exp\pr{C_{q, q_2} M_1^2}}^{\mu} } + \frac{c}{\log d} \set{C_{t,t_0} M_1^2 + \log \brac{\frac{C_{q,t} \hat C\pr{1 + C_{q_2}M_2}}{ \hat c \sqrt{d-1}}}}},
\end{align*}
where $\mu = \frac{tq}{tq - q-t}$.
Since $\norm{v }_{L^t\pr{B_{r}}} \le \exp\pr{C_{t,t_1} M_1^2} \norm{\gr u}_{L^{t_1}\pr{B_r}}$ for any $t_1 > t$, then we reach the conclusion of the theorem after further simplifications.
\end{proof}

\section{Unique continuation at infinity estimates}
\label{UCEstS}

Here we use Theorem \ref{OofV1} combined with an iterative argument to prove Theorem \ref{LandisThm}.
Our arguments are similar to those that appear in \cite{DKW19} and \cite{Dav19a}, which were inspired by the work of \cite{Dav14} and \cite{LW14}.
We prove the theorem for $q_1 > 2$ and $q_1 = 2$ in slightly different ways, and therefore divide this section accordingly.

\subsection{The case of $q_1 > 2$}

The proof of the theorem relies on an iteration scheme.
Therefore, we begin by presenting two propositions that are instrumental to this argument.
The first proposition gives the initial estimate, while the second gives the iterative step.
The initial estimate is as follows.

\begin{prop}[Initial estimate]
\label{bcProp}
Assume that for some $q_1, q_2 \in \pb{2, \iny}$, $c_0, \de_0 > 0$, $W = W_1 + i W_2 : \R^2 \to \C^2$ satisfies \eqref{W1Cond} and \eqref{W2Cond}.
Let $u: \R^2 \to \C$ be a solution to \eqref{ePDE} for which \eqref{uBd} and \eqref{normed} hold.
For any $\eps_0 > 0$ and any $S \ge S_b\pr{R_0, C_0, c_0, q_1, q_2, \de_0, t_0, \eps_0}$, it holds that
\begin{align}
\inf_{\abs{z_0} = S}\norm{\gr u}_{L^2\pr{B_{1/2}\pr{z_0}}} \ge \exp\pr{- S^\al},
\label{step1Est}
\end{align}
where $\al =\frac{2 \hat q\pr{\check q - 2}} {\check q\pr{\hat q- 2}} +\eps_0$ with $\hat q = \min\set{q_1, q_2}$ and $\check q = \max\set{q_1, q_2}$.
\end{prop}

\begin{proof}
Let $\eps_0 > 0$ be given.
Assume that $S$ is sufficiently large with respect to $R_0$, $C_0$, $c_0$, $q_1$, $q_2$, $\de_0$, $t_0$, $\eps_0$ as we will specify below.
Choose $z_0 \in \R^2$ so that $\abs{z_0} = S-1$.
Define 
\begin{align*}
& \tilde u\pr{z} = u\pr{z_0 + S z} \\
& \widetilde W\pr{z} = S \, W\pr{z_0 + S z}.
\end{align*}
Then $\LP \tilde u - \widetilde W \cdot \gr \tilde u = 0$ in $B_2$.
Assumption \eqref{W1Cond} implies that 
\begin{align*}
\norm{\widetilde W_{1}}_{L^{q_1}\pr{B_2}}
&\le S \pr{\int_{\R^2} \abs{W_1\pr{z_0 + S z}}^{q_1} dz}^{1/{q_1}}
= S^{1 - \frac 2 {q_1}},
\end{align*}
while \eqref{W2Cond} implies that $\norm{W_2}_{L^{q_2}\pr{\R^2}} \le A\pr{c_0, \de_0}$, from which it follows that
\begin{align*}
\norm{\widetilde W_{2}}_{L^{q_2}\pr{B_2}}
&\le S \pr{\int_{\R^2} \abs{W_2\pr{z_0 + S z}}^{q_2} dz}^{1/{q_2}}
\le A S^{1 - \frac 2 {q_2}}.
\end{align*}
We see that 
\begin{align*}
\norm{\widetilde W}_{L^{\hat q}\pr{B_2}} 
&\le \norm{\widetilde W_{1}}_{L^{\hat q}\pr{B_2}} + \norm{\widetilde W_{2}}_{L^{\hat q}\pr{B_2}} 
\le C_{\hat q,q_1}\norm{\widetilde W_{1}}_{L^{q_1}\pr{B_2}} + C_{\hat q,q_2}\norm{\widetilde W_{2}}_{L^{q_2}\pr{B_2}} \\
&\le C_{\hat q,q_1}S^{1 - \frac 2 {q_1}} + C_{\hat q,q_2}A S^{1 - \frac 2 {q_2}}.
\end{align*}
Moreover, $\norm{\tilde u}_{L^\iny\pr{B_2}} \le \exp\brac{C_0\pr{3S}^{1 - \frac 2 {q_1}}}$ and from \eqref{normed} we have
\begin{align*}
c_{t_0} \norm{\gr \tilde u}_{L^{2}\pr{B_1}} 
\ge \norm{\gr \tilde u}_{L^{t_0}\pr{B_1}} 
\ge S\norm{\gr u}_{L^{t_0}\pr{B_1\pr{0}}} 
\ge S.
\end{align*}
Observe that
\begin{equation*}
\log\set{\exp\brac{C_0\pr{3S}^{1 - \frac 2 {q_1}}}\frac{{1 + } C_{\hat q,q_1}S^{1 - \frac 2 {q_1}} + C_{\hat q,q_2}A S^{1 - \frac 2 {q_2}}}{S}}\le CS^{1 - \frac 2 {q_1}}.
\end{equation*}
Since $\hat q > 2$, then an application of Theorem \ref{OofV0} shows that
\begin{align*}
\norm{\gr u}_{L^{2}\pr{B_{1/2}\pr{z_0}}}
&=\frac 1S \norm{\gr \tilde u}_{L^{2}\pr{B_{1/2S}}}
\ge \pr{\frac 1 {2S}}^{C \pr{C_{q,q_1}S^{ \frac {q_1 -2} {q_1}} + C_{q,q_2}A S^{\frac{q_2 - 2} {q_2}}}^{\frac{2\hat q}{\hat q-2}}}
\ge \exp\pr{- C S^{\frac{2 \hat q\pr{\check q - 2}} {\check q\pr{\hat q- 2}}} \log S},
\end{align*}
where we have assumed that $S$ is large with respect to $C_0$, $q_1$, $q_2$, and $A$.
Assuming further that $S$ is so large that $C \log S \le S^{\eps_0} \pr{1 - \frac 1 {S}}^\al$, we see that \eqref{step1Est} holds, as required.
\end{proof}

Now we present the proposition which will be repeatedly applied in the proof of Theorem \ref{LandisThm} when $q_1 > 2$.

\begin{prop}[Iterative estimate]
\label{itProp}
Assume that for some $q_1, q_2 \in \pb{2, \iny}$, $c_0, \de_0 > 0$, $W = W_1 + i W_2 : \R^2 \to \C^2$ satisfies \eqref{W1Cond} and \eqref{W2Cond}.
Let $u: \R^2 \to \C$ be a solution to \eqref{ePDE} for which \eqref{uBd} holds.
Let $\eps > 0$, $\eps_1 \in \pr{0, \frac{\de_0}{1 - \frac 2 {q_1} + \de_0}}$.
Suppose that for any $S \ge S_r\pr{R_0, C_0, c_0, q_1, q_2, \de_0, \eps_1, \eps}$, there exists an $\al > 1 + \eps$ so that
\begin{align}
\inf_{\abs{z_0} = S}\norm{\gr u}_{L^2\pr{B_{1/2}\pr{z_0}}} \ge \exp\pr{- S^\al}.
\label{stepnEst}
\end{align}
With $R = S + \pr{\frac S 2}^{\frac1{1- \eps_1}} -\frac 1 2$ and $\be = \left\{\begin{array}{ll} \al - \frac{\al -1}{2}\eps_1 & \text{ if } \al\pr{1 - \eps_1} > 1 - \frac 2 {q_1} \\ 1 - \frac 2 {q_1} + 2 \eps_1 & \text{ otherwise} \end{array}\right.$, it holds that
\begin{equation}
\inf_{\abs{z_1} = R} \norm{\gr u}_{L^2\pr{B_{1/2}\pr{z_1}}} \ge \exp\pr{- R^\be}.
\label{stepn1Est}
\end{equation}
\end{prop}

\begin{proof}
Define $T = \pr{\frac S 2}^{\frac 1{1 - \eps_1}}$ and set $d = 1 + \frac{S}{2T}$. 
Let $z_1 \in \R^2$ be such that $\abs{z_1} = S + T -\frac 1 2= R$.
Define 
\begin{align*}
&\tilde u\pr{z} = u\pr{z_1 + T z} \\
&\widetilde W\pr{z} = T W\pr{z_1 + T z}.
\end{align*}
Then $\LP \tilde u - \widetilde W \cdot \gr \tilde u = 0$ in $B_d$.
Assumption \eqref{W1Cond} implies that 
\begin{align*}
\norm{\widetilde W_1}_{L^{q_1}\pr{B_d}}
&\le T \pr{\int_{\R^2} \abs{W_1\pr{z_1 + T z}}^{q_1} dz}^{1/{q_1}}
= T^{1 - \frac 2 {q_1}},
\end{align*}
while 
\begin{align*}
\norm{\widetilde W_2}_{L^{q_2}\pr{B_d}}
&= T \pr{\int_{B_d} \abs{W_2\pr{z_1 + T z}}^{q_2} dz}^{1/{q_2}}
= T^{1 - \frac 2 {q_2}} \pr{\int_{B_{Td}\pr{z_1}} \abs{W_2\pr{z}}^{q_2} dz}^{1/{q_2}}.
\end{align*}
We may cover $B_{Td}\pr{z_1}$ with $N \sim T^2$ balls of radius $1$, so it follows from condition \eqref{W2Cond} that
\begin{align*}
\norm{\widetilde W_2}_{L^{q_2}\pr{B_d}}
&\le T^{1 - \frac 2 {q_2}} \pr{\sum_{j=1}^N \int_{B_{1}\pr{z_j}} \abs{W_2\pr{z}}^{q_2} dz}^{1/{q_2}}
\le T^{1 - \frac 2 {q_2}} \brac{\sum_{j=1}^N \exp\pr{- q_2 c_0 \abs{z_j}^{1 -\frac 2 {q_1} + \de_0}}}^{1/{q_2}} \\
&\le T^{1 - \frac 2 {q_2}} \set{cT^2 \exp\brac{- q_2 c_0 \pr{\frac {S-1} 2}^{1 - \frac 2 {q_1} + \de_0}}}^{1/{q_2}}
\le \exp\pr{- \tilde c_0 S^{1 - \frac 2 {q_1} + \de_0}},
\end{align*}
where we have used that each ball is centered a distance of at least ${\frac {S-1} 2}$ from the origin.
Moreover, $\norm{\tilde u}_{L^\iny\pr{B_d}} \le \exp\brac{C_0 \pr{\frac 3 2 S + 2T}^{1 - \frac 2 {q_1}}} \le \exp\pr{5^{1 - \frac 2 {q_1}} C_0 T^{1 - \frac 2 {q_1}}} = \exp\pr{\tilde C_0 T^{1 - \frac 2 {q_1}}}$ and from \eqref{stepnEst} we see that with $z_0 := S \frac{z_1}{\abs{z_1}}$, 
\begin{align*}
\norm{\gr \tilde u}_{L^2\pr{B_1}} 
\ge T \norm{\gr u}_{L^2\pr{B_{1/2}\pr{z_0}}} \ge \exp\pr{- {c} S^\al }.
\end{align*}
We are now in a position to apply Theorem \ref{OofV1} to the function $\tilde u$.
Doing so yields 
\begin{align*}
\norm{\gr \tilde u}_{L^2\pr{B_{1/2T}\pr{0}}} 
&\ge \pr{\frac 1 {2T}}^{C_{2} \brac{1+ \exp\pr{C_3 T^{1 - \frac 2 {q_1}}- \tilde c_0 \mu_2 S^{1 - \frac 2 {q_1} + \de_0}}} + \frac{2c T}{S} \brac{ \tilde C_1 T^{1 - \frac 2 {q_1}} + {c} S^\al + \exp\pr{- \tilde c_0 S^{1 - \frac 2 {q_1} + \de_0}} + \log \pr{C_2\sqrt{\frac{2T}{S}}}}},
\end{align*}
where $\tilde C_1 = \tilde C_0 + C_1$, $\mu_2 = \frac{2q_2}{q_2-2}$ and all of the new constants depend on $R_0$, $q_1$, and $q_2$.
If $S$ is sufficiently large in the sense that $\tilde c_0 \mu_2 S^{1 - \frac 2 {q_1} + \de_0} \ge C_3 \pr{S/2}^{\frac {1 - \frac 2 {q_1}}{1 -\eps_1}}$ (which is always possible because of the relationship between $\eps_1$ and $\de_0$), then
\begin{align*}
\norm{\gr u}_{L^2\pr{B_{1/2}\pr{z_1}}} 
&= \frac 1T \norm{\gr \tilde u}_{L^2\pr{B_{1/2T}\pr{0}}} 
\ge \pr{\frac 1 {2T}}^{2C_{2} + \frac{2 \tilde c T}{S} \pr{ \tilde C_1 T^{1 - \frac 2 {q_1}} + {c}S^\al}}.
\end{align*}
If $\al\pr{1 - \eps_1} > 1 - \frac 2 {q_1}$, then $S^\al > T^{1 - \frac 2 {q_1}}$ and then
$$\norm{\gr u}_{L^2\pr{B_{1/2}\pr{z_1}}}  \ge \exp\pr{-C T^{\al - \pr{\al - 1}\eps_1} \log T}.$$
If $S$ is sufficiently large in the sense that 
$\pr{S/2}^{ \frac{\eps_1\eps}{2\pr{1-\eps_1}}} \ge \frac C {1 - \eps_1}\log\pr{S/ 2}$,
then $R^\be \ge C T^{\al - \pr{\al -1} \eps_1} \log T$ and it follows that
\begin{equation}
\norm{\gr u}_{L^2\pr{B_{1/2}\pr{z_1}}} \ge \exp\pr{- R^\be}.
\label{z1Bound}
\end{equation}
On the other hand, if $\al\pr{1 - \eps_1} \le 1 - \frac 2 {q_1}$, then the first term is dominant and
\begin{align*}
\norm{\gr u}_{L^2\pr{B_{1/2}\pr{z_1}}} 
&\ge \exp\pr{-CT^{1 - \frac 2 {q_1} + \eps_1} \log T}.
\end{align*}
If $S$ is large enough so that
$\pr{S/2}^{\frac{\eps_1}{1 - \eps_1}} \ge \frac{C}{1 - \eps_1}\log \pr{S/2}$,
then we again see that \eqref{z1Bound} holds.
Since $z_1 \in \R^2$ with $\abs{z_1} = R$ was arbitrary, \eqref{stepn1Est} has been shown.
\end{proof}

Now we use Proposition \ref{bcProp} followed by repeated applications of Proposition \ref{itProp} to prove Theorem \ref{LandisThm}.

\begin{proof}[The proof of Theorem \ref{LandisThm} for $q_1 > 2$]
Let $\eps > 0$ be given then choose $\eps_1 \in \pr{0, \min\set{\frac{\de_0}{1 - \frac 2 {q_1} + \de_0}, \frac{\frac 2 {q_1} + \frac \eps 2}{1 + \frac \eps 2}}}$ and $\eps_0 > 0$.
Choose $S_0 \ge \max\set{S_b\pr{R_0, C_0, c_0, q_1, q_2, \de_0, t_0, \eps_0}, S_r\pr{R_0, C_0, c_0, q_1, q_2, \de_0, \eps_1, \frac \eps 2}}$, where $S_b$ and $S_r$ are as given in Propositions \ref{bcProp} and \ref{itProp}, respectively.
Define $\al_0 = \frac{2 \hat q\pr{\check q - 2}} {\check q\pr{\hat q- 2}} +\eps_0$, where $\hat q = \min\set{q_1, q_2}$ and $\check q = \max\set{q_1, q_2}$.
An application of Proposition \ref{bcProp} shows that 
\begin{align*}
\inf_{\abs{z} = S_0}\norm{\gr u}_{L^2\pr{B_{1/2}\pr{z}}} \ge \exp\pr{- S_0^{\al_0}}.
\end{align*}
By assumption, we have that $1 + \frac \eps 2 > \frac{1 - \frac 2 {q_1}}{1 - \eps_1}$.
Assuming that $\al_k > 1 + \frac \eps 2$ for $k = 0, 1, \ldots$, we are in the first case of the choice for $\be$ from Proposition \ref{itProp}, so we recursively define 
\begin{align*}
&\al_{k+1} = \al_k - \frac{\al_k - 1}{2} \eps_1 \\
&S_{k+1} = S_k + \pr{\frac{S_k}2}^{\frac 1 {1 - \eps_1}} - \frac 1 2 .
\end{align*}
Then, for each such $k$, an application of Proposition \ref{itProp} shows that
\begin{equation*}
\inf_{\abs{z} = S_{k+1}} \norm{\gr u}_{L^2\pr{B_{1/2}\pr{z}}} \ge \exp\pr{- S_{k+1}^{\al_{k+1}}}.
\end{equation*}
Observe that $\abs{\al_k - \al_{k+1}} > \frac {\eps \eps_1}{4}$.
Therefore, there exists $M \in \N$ with $ M \le N := \ceil{4\pr{\al_0 - 1 - \frac \eps 2}/{\eps \eps_1}}$ so that $\al_M > 1 + \frac \eps 2$, while $\al_{M+1} \le 1 + \frac \eps 2$.
In particular, for any $R \ge S_{N+1} \ge S_{M+1}$, it holds that
\begin{equation*}
\inf_{\abs{z} = R} \norm{\gr u}_{L^2\pr{B_{1/2}\pr{z}}} 
\ge \exp\pr{- R^{\al_{M+1}}}
\ge \exp\pr{- R^{1 + \frac \eps 2}}.
\end{equation*}
An application of the Caccioppoli inequality described by \eqref{intEst} shows that 
\begin{align*}
\norm{\gr u}_{L^2\pr{B_{1/2}\pr{z}}} \le C \pr{1+ \norm{W_1}_{L^{q_1}} + \norm{W_2}_{L^{q_2}}} \norm{u}_{L^\iny\pr{B_{1}\pr{z}}}
\le C \norm{u}_{L^\iny\pr{B_{1}\pr{z}}}
\le \exp\pr{R^{\frac \eps 2}} \norm{u}_{L^\iny\pr{B_{1}\pr{z}}},
\end{align*}
assuming that $R$ is sufficiently large with respect to $C$.
Combining the previous two inequalities leads to the conclusion of the theorem.
\end{proof}

\begin{rem}
The careful reader may wonder why we have avoided using the second case of the choice for $\beta$, i.e., $\beta=1-\frac{2}{q_1}+2\eps_1$, from Proposition \ref{itProp} in our iteration scheme.
As the initial exponent is greater than $2$, then we must always start in the first case.
Each repeated application of Proposition \ref{itProp} will produce an exponent that is greater than $1$.
Therefore, the only way to move into the second case of $\beta$ is by choosing $\eps_1$ so that $\al \pr{1 - \eps_1} \le 1 - \frac 2 {q_1}$.
Doing so implies that $\eps_1 > \frac 2 {q_1}$, and then the resulting exponent is given by $\be = 1 - \frac 2 {q_1} + 2 \eps_1 > 1 + \eps_1$, which still exceeds $1$.
In other words, the second case of $\beta$ doesn't lead to any improvements, so we have chosen to avoid using this case.
\end{rem}

\subsection{The case of $q_1 = 2$}

Now we consider the case where $W_1$ belongs to the threshold space, $L^2$.
In contrast to the previous cases where $q_1 > 2$, here we only need to run the iteration process twice.

\begin{proof}[The proof of Theorem \ref{LandisThm} for $q_1 = 2$]
Choose $q \in \pr{2, q_2}$.
With $\nu = \frac 1 4 \pr{2 - \max\set{\frac{q}{q-1}, t_0}} > 0$ define $t_j = t_0 + j \nu$ for $j = 1, 2, 3$.
Define $\al > \pr{1 - \frac 2 {q_2}}\frac{t_1 q }{t_1 q- q - t_1} > 2$.
For $\eps \in \pr{0, 1}$ as given, define $\eps_0 = \frac \eps {2\pr{\al -1}}$.

Assume that $S$ is sufficiently large with respect to $R_0$, $C_0$, $q_2$, $c_0$, $\de_0$, $t_0$, $\eps$, as well as $q$, $t_1$, $t_2$, $t_3$, $\al$ (which depend on the other terms), as we will specify below.
Choose $z_0 \in \R^2$ so that $\abs{z_0} = S - 1$.
Define 
\begin{align*}
& u_0\pr{z} = u\pr{z_0 + S z} \\
& W_0\pr{z} = S W\pr{z_0 + S z}.
\end{align*}
Then $\LP u_0 - W_0 \cdot \gr u_0 = 0$ in $B_2$.
Assumption \eqref{W1Cond} implies that 
\begin{align*}
\norm{W_{0,1}}_{L^{2}\pr{B_2}}
&\le S \pr{\int_{\R^2} \abs{W_1\pr{z_0 + S z}}^{2} dz}^{1/{2}}
= 1,
\end{align*}
while \eqref{W2Cond} implies that $\norm{W_2}_{L^{q_2}\pr{\R^2}} \le A\pr{c_0, \de_0}$, from which it follows that
\begin{align*}
\norm{W_{0,2}}_{L^{q_2}\pr{B_2}}
&= S \pr{\int_{\R^2} \abs{W_2\pr{z_0 + S z}}^{q_2} dz}^{1/{q_2}}
\le A S^{1 - \frac 2 {q_2}}.
\end{align*}
Moreover, $\norm{u_0}_{L^\iny\pr{B_2}} \le e^{C_0}$ and from \eqref{normed} we see that
\begin{align*}
\norm{\gr u_0}_{L^{t_0}\pr{B_1}} 
\ge S \norm{\gr u}_{L^{t_0}\pr{B_1\pr{0}}} 
\ge{S}.
\end{align*}
An application of Theorem \ref{OofV1} with $d = 2$ shows that
\begin{align*}
\norm{\gr u}_{L^{t_2}\pr{B_{1/2}\pr{z_0}}}
&= \frac 1S \norm{\gr u_0}_{L^{t_2}\pr{B_{1/2S}}} 
\ge \pr{\frac 1 {2S}}^{C_2 \brac{1+ \pr{A S^{1 - \frac 2 {q_2}}}^{\frac{t_1 q }{t_1 q- q - t_1}} e^{C_3}} +  c C_1 + c \log \brac{\frac{C_2 e^{C_0}}{\RD{S}}\pr{1 + A S^{1 - \frac 2 {q_2}} }}}
\nonumber \\
&\ge \exp\pr{- C S^{\pr{1 - \frac 2 {q_2}}\frac{t_1 q }{t_1 q- q - t_1}} \log S},
\end{align*}
where we have assumed that $S$ is large enough to absorb all of the other terms into the dominant one by making the constant larger.
Assuming further that $S$ is so large that $C \log S \le S^{\al - \pr{1 - \frac 2 {q_2}}\frac{t_1 q }{t_1 q- q - t_1}}\pr{1 - \frac 1 S}^\al$, we see that 
\begin{align}
\norm{\gr u}_{L^{t_2}\pr{B_{1/2}\pr{z_0}}}
&\ge \exp\pr{- \abs{z_0}^{\al}} \quad \text{ whenever } \abs{z_0} >> 1.
\label{baseCase}
\end{align}

Recalling that $\eps_0 = \frac{\eps}{2\pr{\al-1}}$, define $T = \pr{\frac S 2}^{\frac 1{\eps_0}}$ and set $d = 1 + \frac{S}{2T}$. 
Let $z_1 \in \R^2$ be such that $\abs{z_1} = S + T -\frac 1 2= R$.
With 
\begin{align*}
&\tilde u\pr{z} = u\pr{z_1 + T z} \\
&\widetilde W\pr{z} = T W\pr{z_1 + T z},
\end{align*}
we see that $\LP \tilde u - \widetilde W \cdot \gr \tilde u = 0$ in $B_d$.
As in the previous proof, assumption \eqref{W1Cond} implies that $\norm{\widetilde W_1}_{L^{2}\pr{B_d}} \le 1$
while 
\begin{align*}
\norm{\widetilde W_2}_{L^{q_2}\pr{B_d}}
&= T \pr{\int_{B_d} \abs{W_2\pr{z_1 + T z}}^{q_2} dz}^{1/{q_2}}
= T^{1 - \frac 2 {q_2}} \pr{\int_{B_{Td}\pr{z_1}} \abs{W_2\pr{z}}^{q_2} dz}^{1/{q_2}}.
\end{align*}
We may cover $B_{Td}\pr{z_1}$ with $N \sim T^2$ balls of radius $1$, so it follows from condition \eqref{W2Cond} that
\begin{align*}
\norm{\widetilde W_2}_{L^{q_2}\pr{B_d}}
&\le T^{1 - \frac 2 {q_2}} \pr{\sum_{j=1}^N \int_{B_{1}\pr{z_j}} \abs{W_2\pr{z}}^{q_2} dz}^{1/{q_2}}
\le T^{1 - \frac 2 {q_2}} \brac{\sum_{j=1}^N \exp\pr{- q_2 c_0 \abs{z_j}^{\de_0}}}^{1/{q_2}} \\
&\le T^{1 - \frac 2 {q_2}} \set{cT^2 \exp\brac{- q_2 c_0 \pr{\frac {S-1} 2}^{\de_0}}}^{1/{q_2}}
\le \exp\pr{- \tilde c_0 S^{\de_0}},
\end{align*}
where we have used that each ball is centered a distance of at least ${\frac {S-1} 2}$ from the origin.
Moreover, $\norm{\tilde u}_{L^\iny\pr{B_d}} \le e^{C_0}$ and from \eqref{baseCase} we see that with $z_0 := S \frac{z_1}{\abs{z_1}}$, 
\begin{align*}
\norm{\gr \tilde u}_{L^{t_1}\pr{B_1}} 
\ge T \norm{\gr u}_{L^{t_1}\pr{B_{1/2}\pr{z_0}}} \ge \exp\pr{- {c} S^\al }.
\end{align*}
Now we apply the order of vanishing estimate described by Theorem \ref{OofV1} again.
With $t_3$ as defined above and $\mu = \frac{t_3 q }{t_3 q- q - t_3}$, we have
\begin{align*}
\norm{\gr u}_{L^{2}\pr{B_{1/2}\pr{z_1}}} 
&=\frac 1T \norm{\gr \tilde u}_{L^{2}\pr{B_{1/2T}}} \\
&\ge \pr{\frac 1 {2T}}^{C_2 \brac{1+ \exp\pr{C_3- \tilde c_0 \mu S^{\de_0}}} + \frac{2c T}{S} \brac{C_1 + C_0 + {c} S^\al + \exp\pr{-\tilde c_0 S^{\de_0}} + \log \pr{C_2 \sqrt{\frac{2T}{S}}}}} \\
&\ge \exp\pr{- C T^{1 + \pr{\al-1}\eps_0} \log T},
\end{align*}
where we have used that $S$ is large enough to absorb all other terms into the dominant one. 
Further assuming that $\log \pr{\frac S 2} \le \frac {\eps_0} C \pr{\frac S 2}^{\frac{\eps}{ 4\eps_0}} = \frac{\eps}{2C\pr{\al -1}} \pr{\frac S 2}^{\frac{\al -1}{2}}$ shows that $C \log T \le T^{\eps/4}$ from which it follows that $C T^{1 + \pr{\al-1}\eps_0} \log T \le R^{1+\frac{3\eps} 4}$.
As in the previous proof, if $R$ is sufficiently large, then an application of the Caccioppoli inequality shows that
$$
\norm{\gr u}_{L^2\pr{B_{1/2}\pr{z_1}}} 
\le C \pr{1+ \norm{W_1}_{L^{q_1}} + \norm{W_2}_{L^{q_2}}} \norm{u}_{L^\iny\pr{B_{1}\pr{z_1}}}
\le C \norm{u}_{L^\iny\pr{B_{1}\pr{z_1}}}
\le \exp\pr{R^{\frac \eps 4}} \norm{u}_{L^\iny\pr{B_{1}\pr{z_1}}}.
$$
It follows that
$$\norm{u}_{L^\iny\pr{B_{1}\pr{z_1}}} \ge \exp\pr{- R^{1 + \eps}}.$$
Since $z_1$ was an arbitrary point of sufficient distance to the origin, the conclusion of the theorem follows.
\end{proof}

\section{Carleman estimates}
\label{CarEstS}

In this section, we prove the Carleman estimate given by Theorem \ref{Carlp2}.
To do this, we rewrite the operator in polar coordinates then use an eigenvalue decomposition to establish our stated bounds.
The techniques used here are very similar to those that appeared in \cite{DZ19}, \cite{DZ18}, \cite{Dav19a}, \cite{DLW19}, and the references therein.

We use standard polar coordinates in $\mathbb R^2\backslash \{0\}$ by setting $x = r \cos \te$ 
and $y = r \sin \te$, where $r = \sqrt{x^2 + y^2}$ and $\te=\arctan\pr{y/x}$.
With the new coordinate $t=\log r$, we see that
\begin{align*}
\del_x= e^{-t}\pr{ \cos \te \der{}{t} - \sin \te \der{}{\te}}, \quad
\del_y= e^{-t}\pr{ \sin \te \der{}{t} + \cos \te \der{}{\te}}
\end{align*}
so that
\begin{align}
\mathcal{L} := 2 e^{t - i \te} \bar \del &= \del_{t} + i \del_{\te}.
\label{LDef}
\end{align}
The eigenvalues of $\del_{\te}$ are $i k$, $k \in \Z$, with corresponding eigenspace $E_k = \textrm{span}\set{e_k}$, 
where $e_k= \frac 1 {\sqrt{2\pi}}e^{ik \te}$ so that $\norm{e_k}_{L^2\pr{S^1}} = 1$.
For any $v \in L^2\pr{S^1}$, let $P_k v = v_k$ denote the projection of $v$ onto $E_k$.
We remark that the projection operator, $P_k$, acts only on the angular variables.
In particular, $P_k v\pr{t, \theta} = P_k v\pr{t, \cdot} \pr{\theta}$.
We may then rewrite the operator $\mathcal L$ as
\begin{align}
\mathcal{L} = \del_{t} - \sum_{k \in \Z} k P_k.
\label{LExpand}
\end{align}
Changing to the variable $t = \log \abs{z}$, the weight function is given by
$$\vp(t)= t+ \tfrac 1 2 \log t^2.$$

Since our result applies to functions that are supported in $B_{R_0} \setminus \set{0}$, then in terms of the new coordinate $t$, we study the case when $t$ is sufficiently close to $-\infty$. By a slight modification to the result described by \cite[Lemma 2]{DZ18} (see also \cite[Lemma 5.1]{DLW19}), we get the following lemma. For the proof of this result, we refer the reader to either \cite{DZ18} or \cite{DLW19}.

\begin{lem}
\label{upDown}
Let $M, N \in \N$ and let $\set{c_k}$ be a sequence of numbers such that $\abs{c_k} \le 1$ for all $k$.
For any $v \in L^2\pr{S^{1}}$ and every $p \in \brac{1, 2}$, we have that
\begin{align}
\norm{\sum^M_{k=N} c_k P_k v}_{L^2(S^{1})}
&\leq C \pr{ \sum^M_{k=N} |c_k|^2}^{\frac{1}{p}-\frac{1}{2}} \norm{v}_{L^p(S^{1})},
\label{seqBd}
\end{align}
where $C = C\pr{p}$.
\end{lem}

The following proposition is crucial to the proof of Theorem \ref{Carlp2}.

\begin{prop}
\label{CarLp2}
Let $p \in \pb{1, 2}$. 
There exists a $t_0 < 0$ such that for any $\tau\gg 1$ and any $u \in C^\iny_c\pr{(-\infty, \ t_0)\times S^{1}}$, it holds that
\begin{equation}
\norm{t^{-1} e^{-\tau \vp(t) }u}_{L^2(dtd\te)}
\leq C \tau^{-1 + \frac 1 p} \norm{t e^{-\tau \vp(t)} \mathcal{L} u}_{L^p(dtd\te)},
\label{key}
\end{equation}
where $C = C\pr{p, t_0}$.
\end{prop}

\begin{proof}
To prove this lemma, we introduce the conjugated operator $\Lt$ of $\mathcal{L}$, defined by
$$\Lt v=e^{-\tau \vp\pr{t}}\mathcal{L}\pr{e^{\tau \vp\pr{t}} v}.$$
With $u = e^{\tau \vp\pr{t}} v$, inequality \eqref{key} is equivalent to
\begin{equation}
\norm{t^{-1} v}_{L^2(dtd\te)} \leq C \tau^{-1 + \frac 1 p} \norm{t \Lt v}_{L^p(dtd\te)}.
\label{keyb}
\end{equation}

From \eqref{LDef} and \eqref{LExpand}, the operator $\Lt$ takes the form
\begin{equation}
\Lt =\sum_{k \in \Z} (\partial_t + \tau \vp'\pr{t} -k)P_k
=\sum_{k \in \Z} (\partial_t + \tau + \tau t^{-1}  -k)P_k.
\label{ord1}
\end{equation}

We first consider $p = 2$.
Since $\disp \Lt v =\partial_t v + \tau\pr{1 + t^{-1}} v  -\sum_{k} k v_k$, then an integration by parts shows that
\begin{align*}
\norm{\Lt v}_{L^2\pr{dt d\te}}^2
&= \iint \abs{\partial_t v + \tau\pr{1 + t^{-1}} v  -\sum_{k \in \Z} k v_k}^2 dt \, d\te \\
&= \iint \abs{\partial_t v}^2  dt \, d\te
+  \iint \sum_{k} \brac{\tau\pr{1 + t^{-1}}  - k}^2 \abs{v_k}^2 dt \, d\te \\
&+ \iint \tau\pr{1 + t^{-1}} \partial_t \abs{v}^2 dt \, d\te
- \iint \sum_{k \in \Z} k \partial_t \abs{v_k}^2 dt \, d\te 
\ge \tau \norm{ t^{-1} v}_{L^2\pr{dt d\te}}^2,
\end{align*}
which implies \eqref{keyb} when $p = 2$.

Now we consider all $p \in \pr{1, 2}$.
Since $\disp \sum_{k \in \Z} P_k v= v,$ we split the sum into three parts.
Let $M=\lceil 2\tau\rceil$ and define
$$ P^h_\tau=\sum_{k> M}P_k, \quad \quad  P^l_\tau=\sum_{k=0}^{M}P_k , \quad \quad P_\tau^n=\sum_{k< 0} P_k.$$
In order to prove the \eqref{keyb}, it suffices to show that for any $p \in \pr{1, 2}$ {and any $v \in C^\iny_c\pr{(-\infty, \ t_0)\times S^{1}}$}
\begin{equation}
\norm{t^{-1} P^\square_\tau v }_{L^2(dtd\te)}
\le C \tau^{-1 + \frac 1 p} \norm{ t \Lt v }_{L^{{p}}(dtd\te)}
\label{key1}
\end{equation}
for $\square = h, l, n$.
The sum of all three inequalities will yield \eqref{keyb}, which implies \eqref{key}.

From \eqref{ord1}, we have the first order differential equation
\begin{equation*}
P_k \Lt v= \pr{\del_t + \tau \vp'\pr{t}-k}P_k v.
\end{equation*}
For $v \in C^\infty_{c}\pr{ (-\infty, \ t_0)\times S^{1}}$, solving the first order differential equation gives that
\begin{equation}
\label{star}
\begin{aligned}
P_k v(t, \te)
&=-\int_{t}^{\infty} e^{k(t-s)+\tau\pr{\vp(s)- \vp\pr{t}}} P_k \Lt v (s, \te)\, ds \\
&= \int_{-\iny}^{t} e^{k(t-s)+\tau\pr{\vp(s)- \vp\pr{t}}} P_k \Lt v (s, \te)\, ds.
\end{aligned}
\end{equation}

We first establish \eqref{key1} with $\square = h$ using the first line of \eqref{star}.
For $k > M \ge 2 \tau$, if $-\iny < t \le s \le t_0 <0$, then
\begin{equation*}
k(t-s)+\tau\pr{\vp(s)- \vp\pr{t}}
= -\pr{k - \tau}\abs{t-s}+ \frac \tau 2 \log\pr{s^2/t^2}
\le - \frac k2 \abs{t-s}.
\end{equation*}
Taking the $L^2\pr{S^{1}}$-norm in \eqref{star} and using this bound gives that
\begin{equation*}
\norm{P_k v(t, \cdot)}_{L^2(S^{1})}
\le \int_{-\infty}^{\infty} e^{-\frac{1}{2}k|t-s|} \norm{P_k \Lt v(s, \cdot)}_{L^2(S^{1})} \,ds.
\end{equation*}
With the aid of \eqref{seqBd}, we get
\begin{equation*}
\norm{P_k v(t, \cdot)}_{L^2(S^{1})}
\le C \int_{-\infty}^{\infty} e^{-\frac{1}{2}k|t-s|} \norm{\Lt v(s, \cdot)}_{L^p(S^{1})} \,ds
\end{equation*}
for any $1 \le p\le 2$.
Applying Young's inequality for convolution then yields
\begin{equation*}
\norm{P_k v}_{L^2(dtd\te)}
\le C \pr{\int_{-\infty}^{\infty} e^{-\frac{\sigma}{2}k|z|} dz}^{\frac{1}{\sigma}} \norm{\Lt v}_{L^p(dtd\te)}
\le C k^{\frac{1}{p} - \frac{3}{2}} \norm{\Lt v}_{L^p(dtd\te)},
\end{equation*}
where $\frac{1}{\sigma}=\frac{3}{2}-\frac{1}{p}$.
Squaring and summing up $k> M$ gives that
$$\sum_{k> M} \norm{P_k v}_{L^2(dtd\te)}^2
\le C \sum_{k> M} k^{-3+\frac{2}{p}} \norm{\Lt v}_{L^p(dtd\te)}^2
= C \tau^{-2+\frac{2}{p}} \norm{\Lt v}_{L^p(dtd\te)}^2,$$
where we have used that $p > 1$ to conclude that the series converges.
An application of orthogonality shows that
$$\norm{P^h_\tau v }_{L^2(dtd\te)} \le C \tau^{-1 + \frac 1 p} \norm{\Lt v }_{L^{{p}}(dtd\te)}$$
which implies \eqref{key1} with $\square = h$.

Now we prove \eqref{key1} for $\square = n$ using the second line of \eqref{star}.
For $k < 0$, if $-\iny < s \le t \le t_0$, then
\begin{align*}
k(t-s)+\tau\pr{\vp(s)- \vp\pr{t}}
&= -\pr{\tau - k} \abs{t-s} + \tau \log\pr{1 + \frac{\abs{s-t}}{\abs{t}}} \\
&\le -\pr{\frac \tau 2 - k} \abs{t-s},
\end{align*}
where we have performed a Taylor expansion.
Repeating the arguments from above shows that for $k < 0$,
\begin{equation*}
\norm{P_k v}_{L^2(dtd\te)}
\le C \pr{\frac \tau 2 - k}^{\frac{1}{p} - \frac{3}{2}} \norm{\Lt v}_{L^p(dtd\te)}.
\end{equation*}
Squaring and summing up $k < 0$ gives that
$$\sum_{k< 0} \norm{P_k v}_{L^2(dtd\te)}^2
\le C \tau^{-2+\frac{2}{p}} \norm{\Lt v}_{L^p(dtd\te)}^2,$$
where we have again used that $p > 1$ to conclude that the series converges.
As in the previous setting, \eqref{key1} holds with $\square = n$.

Fix $t\in (-\infty, \ t_0)$ and set $N=\lceil \tau \vp'(t)\rceil$.
Recalling that $\vp(t)= t + \tfrac 1 2 \log t^2$, an application Taylor's theorem shows that for all $s, t \in (-\iny, \ t_0)$
\begin{align*}
\vp(s) - \vp(t)
&= \vp'(t)(s-t)+\frac{1}{2}\vp''(s_0)(s-t)^2,
\end{align*}
where $s_0$ is some number between $s$ and $t$.
If $s>t$, then
\begin{align}
k(t-s)+\tau \pr{\vp(s) - \vp(t)}
&\le -\pr{k - N}\abs{t-s} - \frac \tau {2 t^2} \pr{s-t}^2.
\label{Taylorh}
\end{align}
Alternatively, if $s \le t$, then
\begin{align}
k(t-s)+\tau \pr{\vp(s) - \vp(t)}
&\le -\pr{N-1 - k}\abs{t-s}  - \frac \tau {2 s^2} \pr{s-t}^2.
\label{Taylorl}
\end{align}
For this reason, we split the sum corresponding to $\square = l$ and use both representations from \eqref{star}.

First we consider the values $N\leq k\leq M$.
From the first line \eqref{star}, we sum over $k$ and use the bound from \eqref{Taylorh} to get
\begin{equation*}
\norm{ \sum^M_{k=N} P_k v(t, \cdot)}_{L^2(S^{1})}
\le \int_{-\infty}^{\infty} 
\norm{ \sum^M_{k=N} e^{-\pr{k - N}\abs{t-s} - \frac \tau {2t^2} \pr{s-t}^2} P_k \Lt v (s, \cdot) }_{L^2(S^{1})}\, ds.
\end{equation*}
With $c_k= e^{-\pr{k - N}\abs{t-s} - \frac \tau {2t^2} \pr{s-t}^2}$, it is clear that $|c_k|\leq 1$.
Therefore, Lemma \ref{upDown} is applicable, so we may apply estimate \eqref{seqBd} to obtain
\begin{equation*}
\begin{aligned}
&\norm{ \sum^M_{k=N} e^{-\pr{k - N}\abs{t-s} - \frac \tau {2t^2} \pr{s-t}^2} P_k \Lt v (s, \cdot) }_{L^2(S^{1})}  
\le C \pr{\sum^M_{k=N} e^{-\pr{k - N}\abs{t-s} - \frac \tau {2t^2} \pr{s-t}^2}}^{\frac{1}{p}-\frac{1}{2}}
 \norm{ \Lt v(s, \cdot)}_{L^p(S^{1})}
\end{aligned}
\end{equation*}
for all $1 \le p \le 2$.
Since
\begin{align*}
\sum^M_{k=N} e^{-2\pr{k - N}\abs{t-s}}
&\le \sum^\iny_{k=0} e^{-2 k\abs{t-s}}
\le 1+ \abs{t-s}^{-1},
\end{align*}
then
\begin{equation*}
\norm{ \sum^M_{k=N} P_k v(t, \cdot)}_{L^2(S^{1})}
\le C \int_{-\infty}^{\infty} e^{- \frac {\al \tau}{2t^2} \pr{s-t}^2}(1+\abs{t-s}^{-\al}) \norm{ \Lt v(s, \cdot)}_{L^p(S^{1})}\, ds,
\end{equation*}
where $\al =\frac{2-p}{2p}$.
Given that
$$e^{\frac {\al \tau}{2t^2} \pr{s-t}^2} \ge \sqrt{1+\frac {\al \tau}{t^2} \pr{s-t}^2} 
\ge C(t_0) \abs{t}^{-1} (1+\tau^{1/2} |s-t|),$$
then, since $\al > 0$, it follows that
\begin{equation*}
e^{-\frac {\al \tau}{2t^2} \pr{s-t}^2} \lesssim \abs{t} (1+\tau^{1/2} |s-t|)^{-1}.
\end{equation*}
We see that
\begin{equation}
\norm{ \sum^M_{k=N} P_k v(t, \cdot)}_{L^2(S^{1})}
\le C \int_{-\infty}^{\infty}  \frac{(1+\abs{t-s}^{-\al})|t| \norm{ \Lt v(s, \cdot)}_{L^p(S^{1})}}{(1+\tau^{1/2} |s-t|)}\, ds.
\label{midSumHigh}
\end{equation}

For $0 \le k \le N-1$, we use the second line of \eqref{star}, then sum over $k$ and use the bound from \eqref{Taylorl} to get
\begin{equation*}
\norm{ \sum^{N-1}_{k=0} P_k v(t, \cdot)}_{L^2(S^{1})}
\le \int_{-\infty}^{\infty} 
\norm{ \sum^{N-1}_{k=0} e^{-\pr{N-1 - k}\abs{t-s}  - \frac \tau {2s^2} \pr{s-t}^2} P_k \Lt v (s, \cdot) }_{L^2(S^{1})}\, ds.
\end{equation*}
Arguing as before, we similarly conclude that
\begin{equation}
\label{midSumLow}
\norm{ \sum^{N-1}_{k=0} P_k v(t, \cdot)}_{L^2(S^{1})}
\le C \int_{-\infty}^{\infty}  \frac{(1+\abs{t-s}^{-\al})|s| \norm{ \Lt v(s, \cdot)}_{L^p(S^{1})}}{(1+\tau^{1/2} |s-t|)}\, ds.
\end{equation}
Combining \eqref{midSumHigh} and \eqref{midSumLow} shows that
\begin{equation*}
\norm{ t^{-1} P_\tau^l v(t, \cdot)}_{L^2(S^{1})}
\le C \int_{-\infty}^{\infty} \frac{(1+\abs{t-s}^{-\al})\norm{s \Lt v(s, \cdot)}_{L^p(S^{1})}}{ (1+\tau^{1/2} |s-t|)}\, ds.
\end{equation*}
Applying Young's inequality for convolution, we get
\begin{equation*}
\norm{ t^{-1} P_\be^l v}_{L^2(dtd\te)}
\leq C \brac{\int_{-\infty}^{\infty} \pr{\frac{1+\abs{z}^{-\al} }{1+\tau^{1/2} |z|}}^\si \, dz}^{\frac{1}{\si}} 
\norm{t \Lt v}_{L^p(S^{1})},
\end{equation*}
where $\frac{1}{\si}=\frac{3}{2}-\frac{1}{p}$.
A direct calculation then shows that
$$\brac{\int_{-\infty}^{\infty} 
\pr{\frac{1+\abs{z}^{-\al} }{1+\tau^{1/2} |z|}}^\si \, dz}^{\frac{1}{\si}} \le C\tau^{-\frac{1}{2\si}+\frac{\al}{2}}.$$
Since $-\frac{1}{2\si}+\frac{\al}{2} = \frac 1 {2p} - \frac 3 4 + \frac 1 {2p} - \frac 1 4 = - 1 + \frac 1 p$, 
then we have shown \eqref{key1} with $\square = l$, thereby completing the proof of the proposition.
\end{proof}

We now present the proof of Theorem \ref{Carlp2}.

\begin{proof}[Proof of Theorem \ref{Carlp2}]
Since $e^{2t} dt d\te = dz$, then
\begin{align*}
&\norm{t^{-1} e^{-\tau \vp(t) }u}_{L^2(dtd\te)}
= \norm{t^{-1} e^{-\tau \vp(t) -t} u e^{t}}_{L^2(dtd\te)}
= \norm{\pr{r \log r}^{-1} e^{-\tau \phi(r)} u }_{L^2(dz)} \\
&\norm{t e^{-\tau \vp(t)} \mathcal{L} u}_{L^p(dtd\te)}
= \norm{t e^{-\tau \vp(t) -2t/p} 2 e^{t - i \te} \bar \del u e^{2t/p}}_{L^p(dtd\te)}
= 2\norm{r^{1 - 2/p} \pr{\log r} e^{-\tau \phi\pr{r}} \bar \del u}_{L^p(dz)}
\end{align*}
and the result follows from applying Proposition \ref{CarLp2}.
\end{proof}


\end{document}